\documentclass[12pt]{article}
\usepackage{amsmath,amssymb,amsbsy,amsfonts,amsthm,latexsym,
            amsopn,amstext,amsxtra,euscript,amscd,amsthm,url}
\usepackage[margin=3cm]{geometry}
\usepackage[usenames, dvipsnames]{color}
\usepackage[utf8]{inputenc}
\usepackage[english]{babel}
\usepackage{mathrsfs}
\usepackage{caption}
\usepackage{xcolor}
\newtheorem{Thm}{Theorem}
\newtheorem*{Thm*}{Theorem}
\newtheorem{Remark}{Remark}

\newtheorem{lem}{Lemma}
\newtheorem{cor}{Corollary}
\newtheorem{prop}{Proposition}

\DeclareMathOperator*{\res}{Res}
\newcommand{\Ocal}{O}

\global\long\def\epsilon{\varepsilon}

\usepackage{tikz}
\usepackage{amsmath}
\usepackage[most]{tcolorbox}
\DeclareMathAlphabet{\mathpzc}{OT1}{pzc}{m}{it}

\tcbset{colback=black!2!white, colframe=black!50!black, 
        highlight math style= {enhanced, 
            colframe=black,colback=black!2!white,boxsep=0pt}
}
\begin{document}
\date{}

\title{An asymptotic formula for the $2k$-th power mean value of $\left| (L'/L)(1+it_0, \chi)\right|$}
\author{Kohji Matsumoto and Sumaia Saad Eddin}

\maketitle
{\def\thefootnote{}
\footnote{{\it Mathematics Subject Classification (2010)}. 11M06; 11Y35}}

\begin{abstract}
Let $q$ be a positive integer ($\geq 2$), 
$\chi$ be a Dirichlet character modulo $q$, $L(s, \chi)$ be the attached 
Dirichlet $L$-function, 
and let $L^\prime(s, \chi)$ denote its derivative with respect to the complex variable $s$. Let $t_0$ be any fixed real number. The main purpose of this paper is to give an asymptotic formula for the $2k$-th power mean value of $\left|(L^\prime/L)(1+it_0, \chi)\right|$ when $\chi$ runs over all Dirichlet characters modulo $q$ (except the principal character when $t_0=0$).
\end{abstract}

\maketitle

\section{Introduction and Statement of the results} 
Let $q$ be a positive integer, and $s=\sigma+it$ be a complex variable.
Let $\chi$ be a Dirichlet character modulo $q$,
$L(s, \chi)$ be the attached Dirichlet $L$-function, 
and let $L^\prime(s, \chi)$ denote its derivative with respect to $s$. 
The values at $1$ of Dirichlet $L$-functions have received considerable attention, due to their algebraical or geometrical interpretation. Assuming the generalized Riemann hypothesis (GRH), Littlewood~\cite{Little} proved that 
\begin{equation*}
|L(1, \chi)|\leq \left( 2+o(1)\right)e^{\gamma}\log \log q.
\end{equation*}
For infinitely many real characters $\chi$, he also proved that 
\begin{equation*}
|L(1, \chi)|\geq \left( 1+o(1)\right)e^{\gamma}\log \log q.
\end{equation*}
In 1948, Chowla~\cite{Chowla} showed that this latter holds unconditionally. The asymptotic properties for the $2k$-th power mean value of $L$-functions at $s=1$ have been studied by many authors: when $k=1$ and $q=p$ is a prime number by Walum~\cite{W}, Slavutski{\u\i}~\cite{Sla}, \cite{Sla2} and Zhang~\cite{Zh}, \cite{Zh1990}. Walum's proof is based on the Fourier series to evaluate $\sum|L(1, \chi)|^2$ where $\chi$ ranges the odd characters modulo $p$. The sharper asymptotic expansion has been obtained by Katsurada and the first author~\cite{K.M}. For general $k$, Zhang and Wang~\cite{Zh.W} presented an exact calculating formula for the $2k$-th power mean value of $L$-functions with $k\geq 3$.

Less is known about $L^\prime/L$ evaluated also at the point $s=1$, although these values are known to be fundamental in studying the distribution of primes since Dirichlet in 1837.  In this direction of research, using the estimates of the character sums and the Bombieri-Vinogradov theorem, Zhang~\cite{Zh1992} gave an asymptotic formula for 
 $$\sum_{q\leq A}\frac{q}{\varphi(q)}\sum_{\substack{\chi \;{\rm mod}\; q \\ \chi\neq \chi_0}}\left| \frac{L^{\prime}(1, \chi)}{L(1, \chi)}\right|^4$$
for the real number $A>3$, where $\varphi$ is the Euler totient function and
$\chi_0$ denotes the principal character.
  Ihara and the first author~\cite{I.M2014} (using the same argument as in \cite{I.M2011}) gave a result related to the value-distributions of $\{(L^{\prime}/L)(s, \chi)\}_{\chi}$ and of $\{(\zeta^{\prime}/\zeta)(s+i\tau)\}_{\tau}$, where $\chi$ runs over Dirichlet characters with prime conductors and $\tau$ runs over $\mathbb{R}$. Ihara, Murty and Shimura~\cite{I.M.S} studied the maximal absolute value of the logarithmic derivatives $(L^{\prime}/L)(1, \chi)$. \\
  Assuming the GRH, they showed that 
 $$\max_{\substack{\chi \;{\rm mod}\; p \\ \chi\neq \chi_0}}\left| \frac{L^{\prime}(1, \chi)}{L(1, \chi)}\right|\leq (2+o(1))\log \log p, $$ 
 where $p$ is a prime.
Unconditionally, they proved, for any $\varepsilon>0$, that 
\begin{equation}
\label{I.M.S}
\frac{1}{|X_p|}\sum_{\substack{\chi \;{\rm mod} \;p \\ \chi\neq \chi_0}}\left|\frac{L^\prime(1, \chi)}{L(1, \chi)}\right|^{2k}
= 
\sum_{m\geq 1}\frac{\left(\sum\limits_{m=m_1 m_2\cdots m_k}\Lambda(m_1)\cdots\Lambda(m_k)\right)^2}{m^2}+\Ocal\left(p^{\varepsilon-1}\right),
\end{equation}
where $\Lambda(.)$ denotes the von Mangoldt function. The proof of this result is based on the study of distribution of zeros of $L$-functions. 
In this paper, we give an asymptotic formula for the $2k$-th power mean value of $\left|(L^\prime/L)(1+it_0, \chi)\right|$ when $\chi$ runs over all Dirichlet characters modulo $q$ and any fixed  real number $t_0$. Denote by $\varepsilon$ an arbitrarily small positive number, not necessarily the same at each occurrence. 
Put $Q=(\log q)^2/\log\log q$.   Our result is precisely the following:
\begin{Thm}
\label{Thm1}
Let $\chi$ be a Dirichlet character modulo $q\geq 2$. For any fixed real number $t_0\neq 0$ and an arbitrary positive integer $k$, we have
\begin{multline}
\label{00}
\frac{1}{\varphi(q)}\sum_{\chi \mkern3mu \mathrel{\textsl{mod}} \mkern3mu q}\left|\frac{L^\prime(1+it_0, \chi)}{L(1+it_0, \chi)}\right|^{2k}
= 
\sum_{\substack{m\geq 1\\ (m,q)=1}}\frac{\left(\sum\limits_{m=m_1 m_2\cdots m_k}\Lambda(m_1)\cdots\Lambda(m_k)\right)^2}{m^2}\\
+\Ocal\left(\frac{1}{q}\left(\log q\right)^{4k+4}+
\left(\log \left(q(q+|t_0|+2)\right)\right)^{2k}
\exp\left(-\frac{B_1(\log q)^2}{\log(q+|t_0|+2)}\right)
+\frac{1}{\varphi(q)}Z_{k,t_0}(q)
\right),
\end{multline}
where
\begin{align}
\label{estfinal_Z_inthm}
Z_{k,t_0}(q)=
\left\{
\begin{array}{ll}
\Ocal\left((\log q)^{2k-2}e^{-B_2|t_0|}Q^{2k}\right) & (|t_0|>1),\\
\displaystyle{\Ocal\left(\left((\log q)^{2k-2}+\frac{1}{|t_0|^{k-1}}\right)
\frac{Q^k}{|t_0|}\left(Q^k+\frac{1}{|t_0|^k}\right) 
\right)} & (0<|t_0|\leq 1)
\end{array}\right.
\end{align}
with certain positive constants $B_1$ and $B_2$.
\end{Thm}

As we will see in the proof of the theorem, the exponential factor in the above error term
is $\leq q^{-1}$ when $q\geq |t_0|+2$ (see Subsection \ref{subsec5-3}).   
Therefore, noting $\varphi(q)\gg q/\log\log q$, we see that the error term tends to
$0$ as $q\to\infty$ while $t_0$ is fixed.

\begin{Thm}
\label{cor1}
Let $\chi$ be a Dirichlet character modulo $q\geq 2$. For an arbitrary positive integer $k$, we have
\begin{align}
\label{01}
\frac{1}{\varphi(q)}\sum_{\substack{\chi \mkern3mu \mathrel{\textsl{mod}} \mkern3mu q\\ \chi\neq \chi_0}}\left|\frac{L^{\prime}(1, \chi)}{L(1, \chi)}\right|^{2k}
&= 
\sum_{\substack{m\geq 1 \\(m, q)=1}}\frac{\left(\sum\limits_{m=m_1 m_2\cdots m_k}\Lambda(m_1)\cdots\Lambda(m_k)\right)^2}{m^2}\\
&+\Ocal\left(\frac{(\log q)^{4k+4}}{q}+\frac{1}{\varphi(q)}Z_{k,0}(q)\right),\notag
\end{align}
with
$$
Z_{k,0}(q)
=\Ocal\left((\log q)^{4k}Q^{2k}+
\delta_1\exp\left(-B_3(1-\beta_1)(\log q)^2\right)(1-\beta_1)^{-2k}\right),
$$
where $B_3$ is a certain positive constant, $\beta_1$ denotes the Siegel zero
(defined just after the statement of Proposition \ref{thm3}), and $\delta_1=1$ if $\beta_1$ exists, and $=0$
otherwise.
\end{Thm}

It is worth mentioning that the condition $(m,q)=1$ in the main term in Eqs.~\eqref{00} and \eqref{01} is omitted in the case when $q$ is a prime number (see Remark \ref{rem1} at 
the end of Section~\ref{section5}). \\

Siegel's theorem (see \cite[Corollary~11.15]{Montgomery}) implies that
$1-\beta_1\gg q^{-\varepsilon}$.   Using this estimate we have
$$
\delta_1\exp\left(-B_3(1-\beta_1)(\log q)^2\right)(1-\beta_1)^{-2k}
\leq \delta_1 (1-\beta_1)^{-2k} \ll q^{2k\varepsilon},
$$
which gives the same estimate as Eq.~\eqref{I.M.S}.
Theorem \ref{cor1} provides an refinement (and a generalization to the case of general modulus $q$) on Eq.~\eqref{I.M.S}. In fact, when $q=p$ is a prime, it is shown in \cite{I.M.S} that the factor $p^{\varepsilon}$ in the error term in Eq.~\eqref{I.M.S} can be replaced by a certain $\log$-power under the assumption of the GRH. Our result gives a same type of improvement 
under the much weaker assumption that the Siegel zero does not exist. 
Another merit of our present method is that we can show the mean value formula not only
at the point $s=1$, but at any point on the line $\Re s=1$ (Theorem \ref{Thm1}).
\\

As a consequence of our main results, we show that the values $\left| (L'/L)(1+it_0, \chi)\right|^2$ behave according to a distribution law. It can be formulated as follows. 
\begin{Thm}
\label{Thm2}
There exists a unique probability measure $\mu=\mu(t_0)$ such that for any positive integer $k$, we have  
\begin{equation*}
\frac{1}{p-1}\sum_{\chi \mkern3mu \mathrel{\textsl{mod}} \mkern3mu p}^{\quad\  \prime} \left|\frac{L^\prime(1+it_0, \chi)}{L(1+it_0, \chi)}\right|^{2k} \mathrel{\mathop{\longrightarrow}\limits_{p \rightarrow +\infty}}
\int\limits_{0}^{+\infty}v^{k}\, d\mu (v),
\end{equation*}
where $\sum_{\chi \mkern3mu \mathrel{\textsl{mod}} \mkern3mu p}^{\prime}$  denotes the summation over all  characters $\chi$ modulo  $p$  with $p$ a prime number (expect the principal character in the case $t_0=0$). 
\end{Thm}
This is an existence (and unicity) result, but getting an actual description of $\mu$ is still a tantalizing problem. It is likely to have a geometrical or arithmetical interpretation, on which our approach gives, so far, no information. 
If $\mu$ is absolutely continuous, then there exists a Radon-Nikod{\'y}m density function
for $\mu$, which may be regarded as a kind of ``$M$-function'' in the sense of
\cite{Iha08} \cite{I.M2014}.

Here is a plot of the distribution function 
\begin{equation}
\label{D}
D_q(v,t_0)= \frac{1}{\varphi(q)} \ \#^{\prime} \left\{ \chi\mkern3mu \mathrel{\textsl{mod}} \mkern3mu q \ ; \ \left|\frac{L^\prime(1+it_0, \chi)}{L(1+it_0, \chi)}\right|^2 \leq v\right\},
\end{equation}
for $q=59, 101$ and $257$ and $t_0=0$. The symbol $\#^{\prime}$ denotes the number of Dirichlet characters modulo $q$ satisfying the condition $\left|(L^\prime/L)(1+it_0, \chi)\right|^2 \leq v$ except the principal character in the case $t_0=0.$

\begin{figure}[h]
\centering
\includegraphics{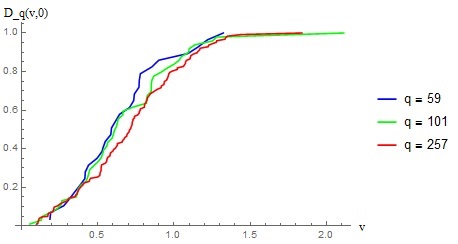}
\caption{The distribution function $D_q(v,0)$.}
\label{Fig1}
\end{figure}
In order to prove our main results, we first prepare several necessary tools in Sections~\ref{Sec2} and \ref{Sec3}. 
\section{Some well-known results}
\label{Sec2}

\begin{prop}
\label{thm1}
Let $m$, $n$, $q$ be positive integers, with $(n, q)=1$. Then we have 
\begin{equation*}
\sum_{\chi \mkern3mu \mathrel{\textsl{mod}} \mkern3mu q }\chi(m)\overline{\chi}(n)=
\begin{cases}
\displaystyle{\varphi(q)} \quad  &\textrm{when $m\equiv n(\mkern3mu \mathrel{\textsl{mod}} \mkern3mu q) $}\\
  \displaystyle{0} \quad  &\textrm{otherwise,}
\end{cases}
\end{equation*}
where the sum is over all characters $\chi (\mkern3mu \mathrel{\textsl{mod}} \mkern3mu q)$.   
\end{prop}
\begin{proof}
See~\cite[Theorem~4.8]{Montgomery}.
\end{proof}

\begin{prop}
\label{thm3}
Let $q\geq 1$.
There is an effectively computable absolute positive constant $c_0$  such that 
$$ \prod_{\chi \mkern3mu \mathrel{\textsl{mod}} \mkern3mu q} L(s, \chi)$$
 has at most one zero $\beta_1$ in the region
$$\sigma \geq 1-\frac{c_0}{\log (q (|t|+2))}. $$

Such a zero, if it exists, is real, simple and corresponds to a non-principal
real character $\chi_1$.
\end{prop}
\begin{proof}
A proof of this theorem can be found in~\cite[Theorem 11.3]{Montgomery}.
\end{proof}
From now on, if $\beta_1$ lies in the following (even smaller) region
\begin{equation}
    \label{region}
    \sigma \geq 1-\frac{c_0}{2\log \left(q(|t|+2)\right)},
\end{equation}
we call $\beta_1$
the exceptional zero (the Siegel zero) and $\chi_1$ the associated character.
\begin{prop}
\label{thm4}
Let $q\geq 1$.   There is an effectively computable positive constant $c\;(<c_0/2)$, which is independent of $q$, for which in the region 
$$\sigma \geq 1-\frac{c}{\log q(|t|+2)}\geq \frac{3}{4}$$
the following estimates hold:
\begin{align}
& \frac{L^{\prime}(s, \chi)}{L(s, \chi)}=\Ocal \left( \log q(|t|+2)\right), \quad \chi\neq \chi_0, \chi_1, \label{pro31}\\
& \frac{L^{\prime}(s, \chi_0)}{L(s, \chi_0)}=-\frac{1}{s-1}+\Ocal \left( \log q(|t|+2)\right), \label{pro32}\\
 &\frac{L^{\prime}(s, \chi_1)}{L(s, \chi_1)}=\frac{1}{s-\beta_1}+\Ocal \left( \log q(|t|+2)\right). \label{pro33}
 \end{align}
\end{prop}
\begin{proof}
A proof of this theorem can be found in~\cite[Kapitel IV, Satz 7.1]{Prachar}.
\end{proof}
\section{Auxiliary lemmas}
\label{Sec3}
\begin{lem}
\label{lem4}
For any integer $m$ and $k\geq 1$, we have
\begin{equation}
\label{eq01}
\sum_{m_1 m_2 \cdots m_k=m}\Lambda(m_1)\cdots \Lambda(m_k)\leq (\log m)^k.
\end{equation}
\end{lem}
\begin{proof}
We prove this lemma by induction on $k$. For $k=1$, it is clear. In order to show that Eq.~\eqref{eq01} is valid for $k=2$, we write
\begin{equation*}
\sum_{m_1 m_2 = m}\Lambda(m_1)\Lambda(m_2)
\le \log m \sum_{ m_2| m}\Lambda(m_2) \leq (\log m)^2.
\end{equation*}
Now, we assume that Eq.~\eqref{eq01} is valid for any fixed and non-negative integer $\ell$ such that $1\leq \ell\leq k-1$. Then we have to prove that it is also valid for $k$. By induction hypothesis, we have   
\begin{eqnarray*}
\sum_{m_1 m_2 \cdots m_k =m}\Lambda(m_1)\cdots \Lambda(m_k)
&=& 
\sum_{m_1n=m}\Lambda(m_1)\sum_{m_2 m_3 \cdots m_k=n}\Lambda(m_2)\cdots \Lambda(m_k)
\\&\leq &
\sum_{m_1n=m}\Lambda(m_1) \log^{k-1}n \leq (\log m)^k.
\end{eqnarray*}
We conclude from the above that Eq.~\eqref{eq01} is valid for $k$. Then it is valid for all $k\geq 1$. The lemma is therefore proved.
\end{proof}
\begin{lem}
\label{lem6}
For any real number $t_0\neq 0$, the Taylor expansion of $(\zeta'/\zeta)(s)$ at $s_0=1+2it_0$ is given by 
\begin{equation}
\label{eq04}
 \frac{\zeta^{'}(s)}{\zeta(s)}=\sum_{n=0}^{\infty}C_{n,s_0} (s-s_0)^n,
\end{equation}
where
\begin{equation}
\label{eq05}
C_{n,s_0}= \Ocal\left(\frac{1}{|t_0|^{n+1}}+\left(\log (|t_0|+2) \right)^{n+1}\right).
\end{equation} 
\end{lem}

\begin{proof}
It is well known that $\zeta(1+it_0)\neq 0$ for every real $t_0\neq 0$, see \cite[Theorem 13.6]{Aps}. Then, the Taylor expansion of $(\zeta'/\zeta)(s)$ at $s_0$ is given by 
\begin{equation*}
\frac{\zeta^{\prime}(s)}{\zeta(s)}=\sum_{n\geq 0}C_{n,s_0} (s-s_0)^n,  
\end{equation*}
where the coefficients $C_{n,s_0}$ are defined by the following residue:
\begin{equation*}
C_{n,s_0}=\res \left(\frac{\zeta^{\prime}(s)}{\zeta(s)}\frac{1}{(s-s_0)^{n+1}}; s_0\right).
\end{equation*}
In order to calculate the residue above, we consider the contour $\mathcal{C}$ 
which is a positively oriented circle of radius $R$ and center $s_0$. 
Proposition \ref{thm3} for $q=1$ gives the classical zero-free region for the 
Riemann zeta-function 
$$\sigma\geq 1-\frac{c_0}{\log (|t|+2)}.$$
We choose $R=c_0/\left(2\log (|t_0|+2)\right)$. Write $s\in \mathcal{C}$ as 
$s=s_0+Re^{i\theta}$, with $0\leq \theta\leq 2\pi$. Here we notice that, when $|t_0|$ is very small, the point $s=1$ may be inside the circle $\mathcal{C}.$ If not, we have 
\begin{equation*}
    C_{n, s_0}=\frac{1}{2\pi i}\int_{\mathcal{C}}\frac{\zeta'(s)}{\zeta(s)}\frac{ds}{(s-s_0)^{n+1}}.
\end{equation*}
Using Eq.~\eqref{pro32}, the integral on the right-hand side is
\begin{eqnarray*}
&=&\frac{1}{2\pi i}\int_{\mathcal{C}}\left( -\frac{1}{s-1}+\Ocal\left(\log(|t|+2)\right)\right)\, \frac{ds}{(s-s_0)^{n+1}}
\\&=&
\Ocal\left( (\log (|t_0|+2))^{n+1}\right).
\end{eqnarray*}
On the other hand, if $s=1$ is inside $\mathcal{C}$, we have
\begin{equation*}
 C_{n, s_0}+\res \left( \frac{\zeta'(s)}{\zeta(s)}\frac{1}{(s-s_0)^{n+1}}; 1\right)=\frac{1}{2\pi i}\int_{\mathcal{C}}\frac{\zeta'(s)}{\zeta(s)}\frac{ds}{(s-s_0)^{n+1}}.
\end{equation*}
It is easy to check that 
\begin{eqnarray*}
\res \left( \frac{\zeta'(s)}{\zeta(s)}\frac{1}{(s-s_0)^{n+1}}; 1\right)&=&\lim_{s\rightarrow 1}\left[(s-1)\frac{\zeta'(s)}{\zeta(s)}\frac{1}{(s-s_0)^{n+1}}\right]
\\&=& \Ocal\left( |t_0|^{-n-1}\right),
\end{eqnarray*}
while the integral term is $\Ocal \left(\left(\log (|t_0|+2) \right)^{n+1}\right)$ (because $|s-1|=|s_0+Re^{i\theta}-1|=|2it_0+Re^{i\theta}|\asymp R\asymp 1$ when $|t_0|$ is small). Lastly, when $s=1$ is on the circle $\mathcal{C}$, we modify $\mathcal{C}$ slightly to obtain the same result. This completes the proof. 
\end{proof}

It is known that the Laurent expansion of the Riemann zeta-function at $s=1$ is given by 
\begin{equation}\label{Laurent}
\zeta(s)=\frac{1}{s-1}+\sum_{n\geq 0}\gamma_n (s-1)^n,
\end{equation}
where $\gamma_n$ are called the Stieltjes constants. 
\begin{lem}
\label{lem7}
We have 
\begin{equation}
\label{eq08}
(s-1)\frac{\zeta'(s)}{\zeta(s)}=\sum_{n\geq 0}E_{n} (s-1)^{n},
\end{equation}
where $E_{0}=-1$ and 
\begin{equation}
\label{eq09}
E_{n}=(n-1)\gamma_{n-1}-\sum_{k=1}^{n}\gamma_{k-1}E_{n-k}\qquad (n\geq 1).
\end{equation}
\end{lem}

\begin{proof}
Differentiating the both sides of \eqref{Laurent}, we have 
\begin{equation*}
\zeta^{\prime}(s)=\frac{-1}{(s-1)^2}+\sum_{n\geq 0}n\gamma_n (s-1)^{n-1}.
\end{equation*}
By making a change of variable and using properties of power series, we find that 
\begin{eqnarray*}
(s-1)\frac{\zeta'(s)}{\zeta(s)}&=&\frac{-1+\sum_{n\geq 0}n \gamma_n (s-1)^{n+1}}{1+\sum_{n\geq 0}\gamma_n (s-1)^{n+1}}
\\&=&
\frac{\sum_{n\geq 0}(n-1)\gamma_{n-1}(s-1)^n}{\sum_{n\geq 0}\gamma_{n-1}(s-1)^n}
\\&=&
\sum_{n\geq 0}E_{n}(s-1)^n,
\end{eqnarray*}
where $ \gamma_{-1}=1$, $E_{0}=-1$ and
$$ E_{n}=(n-1)\gamma_{n-1}-\sum_{k=1}^{n}\gamma_{k-1}E_{n-k}\quad  (n\geq 1).$$
This implies the desired result. 
\end{proof}
\begin{lem}
\label{lem5}
Let $t_0$ be a fixed real number, $p$ be a prime number, and let $a\in\mathbb{C}$
with $\Re a=1$.   The Taylor expansion of the function $\sum_{p|q}(\log p)/\left(p^{s+a}-1\right)$  at the origin is 
\begin{equation}
\label{eq02}
\sum_{p|q}\frac{\log p}{p^{s+a}-1}= \sum_{n=0}^{\infty} {F}_{n,a}s^n,
\qquad F_{n,a}=\Ocal_n(Q).
\end{equation}
\end{lem}

\begin{proof} 
The Taylor expansion of  $(\log p)/\left(p^{s+a}-1\right)$ at the origin is given by 
$$\frac{\log p}{p^{s+a}-1}= \sum_{n\geq 0}F_{n, a}(p)s^n. $$ 
where   
$$ F_{n,a}(p)=\frac{1}{2\pi i} \int_{\mathcal{C}}\frac{\log p}{(p^{s+a}-1)}\frac{ds}{s^{n+1}}.$$
Here, the contour $\mathcal{C}$ is a positively oriented circle of radius $R=1/2$ and centered at the origin. Taking $s=Re^{i\theta}$, where $0\leq \theta\leq 2\pi$, it is easily seen 
(because of the confition $\Re a=1$) that  
\begin{equation*}
F_{n,a}(p)\ll \frac{2^n\log p}{p^{1/2}}.
 \end{equation*}
 Note that the implied constant here is independent of $a$.
 Therefore, we have 
 \begin{equation*}
\sum_{p|q}F_{n,a}(p)\ll_n \sum_{p|q}\frac{\log p}{p^{1/2}}\ll \log q \sum_{p|q}1.
 \end{equation*}
Notice that the latter sum is $\omega(q)$,i.e., the number of distinct prime divisors of $q$.
Using the fact 
${\omega(q)}\ll \log q/\log \log q$ ( see\cite[Theorem 2.10]{Montgomery}),  
we get 
 \begin{equation*}
\sum_{p|q}F_{n,a}(p)=\Ocal_n\left( \frac{(\log q)^2}{\log \log q}\right).
 \end{equation*}
This completes the proof. 
\end{proof}
\begin{lem}
\label{lem8}
Let $\beta_1$ be the Siegel zero corresponding to $\chi_1$.
Then, we have 
\begin{equation*}
\frac{L^{\prime}(s+\beta_1, \chi_1)}{L(s+\beta_1, \chi_1)}=\frac{1}{s}+\sum_{n\geq 0}P_n s^n, 
\qquad  P_n=\Ocal\left((\log q)^{n+1} \right).
\end{equation*}
\end{lem}
\begin{proof}
The Laurent expansion of $(L^{\prime}/L)(s, \chi_1)$ at the point $\beta_1$ is given by 
\begin{equation*}
\frac{L^{\prime}(s, \chi_1)}{L(s, \chi_1)}=\frac{1}{s-\beta_1}+\sum_{n\geq 0}P_n (s-\beta_1)^n,  
\end{equation*}
where the coefficients $P_n$ are defined by 
\begin{equation*}
P_n=\frac{1}{2\pi i}\int_{\mathcal{C}}\frac{L^{\prime}(s, \chi_1)}{L(s, \chi_1)}\frac{ds}{(s-\beta_1)^{n+1}}.
\end{equation*}
Here the contour $\mathcal{C}$ is a positively oriented circle of radius $R=c_2/\log(2q)$ and centered at $\beta_1$, where $ c_2<c_0/2$ is sufficiently small.   We see that the function $(L^{\prime}/L)(s, \chi_1)$ has at most one pole at $s=\beta_1$ that lies inside the circle. Let $s=\beta_1+Re^{i\theta}$ where $ 0\leq \theta\leq 2\pi$. Using Eq.~\eqref{pro33}, we get 
\begin{equation*}
P_n=\frac{1}{2\pi }\int_{0}^{2\pi}\left( \frac{1}{Re^{i\theta}}+\Ocal\left(\log 2q\right)\right)\frac{d\theta}{(Re^{i\theta})^n}
 =
\Ocal\left((\log q)^{n+1}\right).
\end{equation*}
This completes the proof. 
\end{proof}
\begin{lem}
\label{lem9}
Let  $t_0$ be a non-zero real number and let $\beta_1$ be the Siegel zero in the region given by Eq.~\eqref{region} corresponding to a non-principal real character $\chi_1$. Then, the Taylor expansion of the function $(L^{\prime}/L)(s+it_0, \chi_1)$ at the point $s_0=\beta_1+it_0$ is given by 
\begin{equation*}
\frac{L^{\prime}(s, \chi_1)}{L(s, \chi_1)}=\sum_{n\geq 0}Q_n (s-s_0)^n,  
\end{equation*}
where 
\begin{equation*}
Q_n=\Ocal\left((\log q(|t_0|+2))^{n+1}+\frac{1}{|t_0|^{n+1}} \right).
\end{equation*}
\end{lem}
\begin{proof}
The Taylor expansion of $(L^{\prime}/L)(s, \chi_1)$ at the point $ s_0=\beta_1+it_0$ is given by 
\begin{equation*}
\frac{L^{\prime}(s, \chi_1)}{L(s, \chi_1)}=\sum_{n\geq 0}Q_n (s-s_0)^n,  
\end{equation*}
where the coefficients $Q_n$ are defined by 
\begin{equation*}
Q_n=\res \left(\frac{L^{\prime}(s, \chi_1)}{L(s, \chi_1)}\frac{1}{(s-s_0)^{n+1}}; s_0 \right).
\end{equation*}
In order to calculate the residue above, we consider that the contour $\mathcal{C}$ is a positively oriented circle of radius $R=c_3/\log(q(|t_0|+2))$ and centered at $s_0$, where $c_3\leq c_0/2$ is sufficiently small. In the case when $|t_0|$ is very small,  we see that the inside of the contour $\mathcal{C}$ can contain at most one pole of $(L^{\prime}/L)(s, \chi_1)$ at $\beta_1$. Let $s=s_0+Re^{i\theta}$, where $ 0\leq \theta\leq 2\pi$, we find that 
\begin{equation*}
Q_n+\res \left(\frac{L^{\prime}(s, \chi_1)}{L(s, \chi_1)}\frac{1}{(s-s_0)^{n+1}}; \beta_1\right)=\frac{1}{2\pi i}\int_{C}\frac{L^{\prime}(s, \chi_1)}{L(s, \chi_1)}\frac{ds}{(s-s_0)^{n+1}}.
\end{equation*}
Using Eq.~\eqref{pro33}, we get 
\begin{eqnarray*}
\res \left(\frac{L^{\prime}(s, \chi_1)}{L(s, \chi_1)}\frac{1}{(s-s_0)^{n+1}}; \beta_1\right)&=&
\lim\limits_{s\rightarrow \beta_1}\left[(s-\beta_1)\frac{L^{\prime}(s, \chi_1)}{L(s, \chi_1)}\frac{1}{(s-s_0)^{n+1}}\right]
     \\&=&
     \Ocal\left(|t_0|^{-n-1}\right)
\end{eqnarray*}
and 
\begin{equation*}
\frac{1}{2\pi i}\int_{\mathcal{C}}\frac{L^{\prime}(s, \chi_1)}{L(s, \chi_1)}\frac{ds}{(s-s_0)^{n+1}}=\Ocal\left((\log q(|t_0|+2))^{n+1}\right).
\end{equation*}
When $s=\beta_1$ is not inside the circle, the residue term does not appear. This completes the proof. 
\end{proof}
\section{An asymptotic formula}
\label{section4}
To aid in formulating our next result, it is convenient to employ the
notation $m=m_1 m_2\cdots m_k$,  $n=n_1 n_2\cdots n_k$, and $\mathscr{R}$ is a set of the pairs $(m, n)$ with the conditions $ m, n\geq 1$, $ (q, mn)=1$ and $m\equiv n\pmod{q}$. When  we have  extra condition such as  $m=n$, $m\neq n$ or $m<n$, we write $\mathscr{R}_{n=m}$,  $\mathscr{R}_{n\neq m}$ or $\mathscr{R}_{ m<n}$, respectively. 
\begin{prop}
\label{mainpro}
Let $m_i$, $n_i$ and $k$ be positive integers for $i \in \{1, 2, \cdots, k\}$. For any real $t_0$ and $X>1$,  we have
\begin{align}
\label{important}
&\sum_{\mathscr{R}} \frac{\sum\limits_{m=m_1\cdots m_k}\prod\limits_{i=1}^{k}\Lambda(m_i) \sum\limits_{n=n_1\cdots n_k}\prod\limits_{i=1}^{k}\Lambda(n_i)}{m^{1+it_0} n^{1-it_0}} e^{-m n/X}
\\&= \sum_{ \substack{m\geq 1 \\(m,q)=1}} \frac{\Bigl(\sum\limits_{{m}=m_1\cdots m_k}\Lambda(m_1)\cdots \Lambda(m_k)\Bigr)^{2}}{{m}^2}
+\Ocal_k\left(\frac{(\log X)^{2k+2}}{q}+\frac{(\log X)^{2k}}{\sqrt{X}}\right).\nonumber
\end{align}
\end{prop}
\begin{proof}
Without loss of generality we can assume $t_0\geq 0$. In order to prove our proposition, we denote the left-hand side of Eq.~\eqref{important} by $F_q(X)$. 
We split the set $\mathscr{R}$ defined by the condition $m\equiv n(\mkern3mu \mathrel{\textsl{mod}} \mkern3mu q)$ and $(q, mn)=1$ into two subsets.
\begin{itemize}
\item[$\bullet$] The first case is when $(q,m n)= 1$ and $ m \neq n$. We define 
\begin{equation*}
A_q(X):=
 \sum_{\mathscr{R}_{m\neq n}} \Bigl(\sum_{m=\prod\limits_{i=1}^{k} m_i} \Lambda(m_1)\cdots \Lambda(m_k) \sum_{n=\prod\limits_{i=1}^{k} n_i}\Lambda(n_1)\cdots \Lambda(n_k)\Bigr)  \frac{e^{-mn/X}}{m^{1+it_0} n^{1-it_0} }. 
 \end{equation*}
 Applying Lemma~\ref{lem4} to the above, we find that 
 \begin{eqnarray*}
 A_q(X) 
& \ll & 
\sum_{\mathscr{R}_{ m<n}} \frac{e^{-mn/X}}{mn}(\log m)^k (\log n)^k
\\& \ll & 
\sum_{{m}\geq 1}\sum_{\substack{ \ell\geq 1\\ {n}={m}+\ell q}} \frac{e^{-mn/X}}{m n}(\log m)^k (\log n)^k
\\& = &  
\sum_{m\geq 1}\sum_{ \ell\geq 1} \frac{e^{-m(m+\ell q)/X}}{m ({m}+\ell q)}(\log m)^k (\log (m+\ell q))^k
\\& = &
\sum_{m\geq 1} \frac{e^{-m^2/X}(\log m)^k}{m}\sum_{ \ell\geq 1} \frac{e^{-(m\ell q)/X}}{({m}+\ell q)} (\log (m+\ell q))^k.
\end{eqnarray*}
We first estimate the inner sum above as follows:
\begin{eqnarray*}
\sum_{ \ell\geq 1} \frac{e^{-(m\ell q)/X}}{({m}+\ell q)} (\log (m+\ell q))^k &\ll& \int_1^{\infty}\frac{e^{-(mt q)/X}}{({m}+t q)} (\log (m+t q))^k\, dt
\\&\ll &
\left(\int_1^{\frac{X}{mq}}+\int_{\frac{X}{mq}}^{\infty}\right)\frac{e^{-(mt q)/X}}{({m}+t q)} (\log (m+t q))^k\, dt
\\&:=&
I_1+I_2,
\end{eqnarray*}
 say. We notice that $I_1$ does not exist if $m>X/q$. Otherwise, it is estimated by 
\begin{equation*}
I_1\leq \int_{1}^{\frac{X}{mq}}\frac{\left( \log (m+tq)\right)^k}{(m+tq)}\, dt,
\end{equation*}
and putting $m+tq=u$, we have 
\begin{equation}
\label{eqproof1}
I_1\leq \frac{1}{q}\int_{m+q}^{m+\frac{X}{m}}\frac{(\log u)^k}{u}\, du \ll \frac{1}{q}\left( \log \left( m+\frac{X}{m}\right)\right)^{k+1}.
\end{equation}
After making the change of variable $mtq/X=v$, $I_2$ becomes 
\begin{eqnarray*}
I_2&=&\frac{X}{mq}\int_{1}^{\infty}\frac{e^{-v}}{\left(m+\frac{X}{m}v\right)}\left(\log \left( m+\frac{Xv}{m}\right)\right)^k\, dv
\\&\leq&
\frac{1}{q}\int_{1}^{\infty}\frac{e^{-v}}{v}\left(\log \left( m+\frac{Xv}{m}\right)\right)^k \, dv
\\&=&
\frac{1}{q}\left(\int_{1}^{m^2/X}+\int_{m^2/X}^{\infty}\right)\frac{e^{-v}}{v}\left(\log \left( m+\frac{Xv}{m}\right)\right)^k \, dv
\\&\leq &
\frac{(\log  2m)^k}{q}\int_{1}^{m^2/X}\frac{e^{-v}}{v} \, dv+\frac{1}{q}\int_{m^2/X}^{\infty}\frac{e^{-v}}{v}\left(\log  \frac{2Xv}{m}\right)^k \, dv,
\end{eqnarray*}
which yields to 
\begin{equation}
\label{eqproof2}
I_2 \ll 
\frac{1}{q}\left((\log m)^k+(\log X)^k\right).
\end{equation}
From Eqs.~\eqref{eqproof1} and \eqref{eqproof2}, we get 
\begin{equation*}
\sum_{ \ell\geq 1} \frac{e^{-(m\ell q)/X}}{({m}+\ell q)} (\log (m+\ell q))^k 
\ll 
\frac{1}{q}\left((\log m)^k+(\log X)^k+\left(\log \left( m+\frac{X}{m}\right)\right)^{k+1}\right).
\end{equation*}
Therefore
\begin{multline}
\label{eq000125}
qA_q(X)\ll  
\sum_{m\geq 1}\frac{e^{-m^2/X}}{m}(\log m)^{2k}+(\log X)^{k}\sum_{m\geq 1}\frac{e^{-m^2/X}}{m}(\log m)^k
\\+\sum_{m\geq 1}\frac{e^{-m^2/X}}{m}(\log m)^k\left(\log \left( m+\frac{X}{m}\right)\right)^{k+1}.
\end{multline}
The first sum above is  estimated by
\begin{equation*}
 \ll \int_{1}^{\sqrt{X}}\frac{(\log t)^{2k}}{t}\, dt+\int_{\sqrt{X}}^{\infty}\frac{e^{-t^2/X}}{t}(\log t)^{2k}\, dt.
\end{equation*}
The first integral here is estimated by  $\ll (\log X)^{2k+1}$. After making the change of variable $ t^2/X=v$, the second integral is  $\ll (\log X)^{2k} $. This gives us
\begin{equation*}
\sum_{m\geq 1}\frac{e^{-m^2/X}}{m}(\log m)^{2k}\ll (\log X)^{2k+1}.
\end{equation*}
Similarly, we observe that the second term on the right-hand side of Eq.~\eqref{eq000125} is 
\begin{equation*}
\ll (\log X)^{k}(\log X)^{k+1}=(\log X)^{2k+1}.  
\end{equation*}
As for the third sum on the right-hand side of Eq.~\eqref{eq000125}, it is estimated by 
\begin{multline*}
\sum_{m\geq 1}\frac{e^{-m^2/X}}{m}(\log m)^k\left( \log (m+X/m)\right)^{k+1} \\ 
\ll 
(\log X)^k\int_{1}^{\sqrt{X}}\frac{(\log X/t)^{k+1}}{t}\, dt +\int_{\sqrt{X}}^{\infty}\frac{e^{-t^2/X}}{t}\left( \log t\right)^{2k+1}\, dt.
\end{multline*}
It is easy to see that the first integral on the right-hand side of the above is $\ll (\log X)^{2k+2}$. By the change of variable $t^2/X=v$, the second integral is estimated by $\ll (\log X)^{2k+1}$. Thus, we find that 
\begin{equation*}
\sum_{m\geq 1}\frac{e^{-m^2/X}}{m}(\log m)^k\left( \log \left(m+\frac{X}{m}\right)\right)^{k+1} \ll (\log X)^{2k+2}.
\end{equation*}
Therefore, we get 
\begin{equation}
\label{eq012}
A_q(X)\ll \frac{(\log X)^{2k+2}}{q}.
\end{equation}
\item[$\bullet$] The second case is when $(q,m n)= 1$ and $m = n$. Then, define
\begin{equation*}
B_q(X) := \sum_{\mathscr{R}_{m=n}} \Bigl(\sum_{m=\prod\limits_{i=1}^{k} m_i} \Lambda(m_1)\cdots \Lambda(m_k) \sum_{n=\prod\limits_{i=1}^{k} n_i}\Lambda(n_1)\cdots \Lambda(n_k)\Bigr) \frac{e^{-{mn}/X}}{{m}^{1+it_0} {n}^{1-it_0}}, 
\end{equation*}
and put
\begin{equation*}
B_q(X)= B^{\sharp}_q(X)+B^{\flat}_q(X), 
\end{equation*}
where 
\begin{equation*}
B^{\sharp}_q(X) := \sum_{\substack{\mathscr{R}_{m=n}\\ m \leq X^{1/2} }} \Bigl(\sum_{m=\prod\limits_{i=1}^{k} m_i} \Lambda(m_1)\cdots \Lambda(m_k)\Bigr)^2 \frac{e^{-{{m}}^2/X}}{{m}^2}, 
\end{equation*}
and 
\begin{equation*}
B^{\flat}_q(X) := \sum_{\substack{\mathscr{R}_{m=n}\\ m > X^{1/2} }} \Bigl(\sum_{m=\prod\limits_{i=1}^{k} m_i} \Lambda(m_1)\cdots \Lambda(m_k)\Bigr)^2 \frac{e^{-{{m}}^2/X}}{{m}^2}. 
\end{equation*}
For the function $B^{\flat}_q(X)$, since ${m}>X^{1/2}$, we see that  $e^{-{{m}}^2/X} \leq 1$ and  
\begin{eqnarray*}
B^{\flat}_q(X)
&\ll& 
\sum_{\substack{\mathscr{R}_{m=n}\\m > X^{1/2} }}  \frac{\Bigl(\sum_{m=\prod\limits_{i=1}^{k} m_i} \Lambda(m_1)\cdots \Lambda(m_k)\Bigr)^2}{{m}^2}
\\&\ll&
\sum_{m> X^{1/2}} \frac{(\log m)^{2k}}{m^2},
\end{eqnarray*}
where we used Lemma~\ref{lem4}. Thus
\begin{equation}
\label{eq013}
B^{\flat}_q(X)
\ll
\frac{(\log X)^{2k}}{X^{1/2}}.
\end{equation}
For the function $B^{\sharp}_q(X)$, since $ {m}^2$ is small enough, we can rely on the approximation 
\begin{equation*}
e^{-{m}^2/X}=1+\Ocal{\left(\frac{{m}^2}{X}\right)}.
\end{equation*}
Then, the function $B^{\sharp}_q(X)$ is rewritten as 
\begin{equation*}
B^{\sharp}_q(X)=\sum_{\substack{\mathscr{R}_{m=n}\\ m \leq  X^{1/2} }} \frac{\Bigl(\sum\limits_{m=m_1\cdots m_k} \prod\limits_{i=1}^{k}\Lambda(m_i) \Bigr)^2}{{m}^2}+ \Ocal\left(\frac{1}{X}\sum_{\substack{\mathscr{R}_{m=n}\\ m\leq X^{1/2} }} \left(\sum\limits_{m=m_1\cdots m_k} \prod\limits_{i=1}^{k}\Lambda(m_i) \right)^2\right). 
\end{equation*}
Again using Lemma~\ref{lem4}, we see that the error term is $\Ocal\left( X^{-1/2}(\log X)^{2k}\right).$ Further, we remove the condition $m\leq X^{1/2}$ from the summation with the error $\Ocal\left(X^{-1/2}(\log X)^{2k}\right).$
Thus, we have
 \begin{equation}
 \label{eq014}
 B^{\sharp}_q(X)
 =
\sum_{\mathscr{R}_{m=n}}  \frac{\Bigl(\sum\limits_{m=m_1\cdots m_k} \prod\limits_{i=1}^{k}\Lambda(m_i) \Bigr)^2}{{m}^2}+ \Ocal\left(\frac{(\log X)^{2k}}{\sqrt{X}}\right).
 \end{equation}
From Eqs.~\eqref{eq013} and~\eqref{eq014}, we find that 
\begin{equation}
\label{eq015}
B_q(X)=\sum_{\substack{m\geq 1\\(m, q)=1}} \frac{\Bigl(\sum\limits_{m=m_1\cdots m_k} \prod\limits_{i=1}^{k}\Lambda(m_i) \Bigr)^2}{{m}^2}+\Ocal\left(\frac{(\log X)^{2k}}{\sqrt{X}}\right).
\end{equation}
From Eqs.~\eqref{eq012} and \eqref{eq015}, we obtain the assertion of the proposition.
\end{itemize}
\end{proof}
In the case when $q=p$ is a prime number, Proposition~\ref{mainpro} becomes
\begin{prop}
\label{mainpro2}
Let $m_i$, $n_i$ and $k$ be positive integers for $i \in \{1, 2, \cdots, k\}$. Let $q=p$ be a prime number. For any real $t_0$ and $X>1$,  we have
\begin{align}
\label{important1}
&\sum_{\mathscr{R}} \frac{\sum\limits_{m=m_1\cdots m_k}\prod\limits_{i=1}^{k}\Lambda(m_i) \sum\limits_{n=n_1\cdots n_k}\prod\limits_{i=1}^{k}\Lambda(n_i)}{m^{1+it_0} n^{1-it_0}} e^{-m n/X}
\\&= \sum_{ m\geq 1} \frac{\Bigl(\sum\limits_{{m}=m_1\cdots m_k}\Lambda(m_1)\cdots \Lambda(m_k)\Bigr)^{2}}{{m}^2}
+\Ocal_k\left(\frac{(\log X)^{2k+2}}{p}+\frac{(\log X)^{2k}}{\sqrt{X}}+\frac{(\log p)^{2k}}{p^2}\right).\nonumber
\end{align}
\end{prop}
\begin{proof}
This is clear from 
\begin{eqnarray}
\label{eq0152}
\sum_{\substack{ m\geq 1\\ (m,p)=1 }} \frac{\Bigl(\sum\limits_{m=m_1\cdots m_k} \prod\limits_{i=1}^{k}\Lambda(m_i) \Bigr)^2}{{m}^2}&=&
\sum_{ m\geq 1} \frac{\Bigl(\sum\limits_{m=m_1\cdots m_k} \prod\limits_{i=1}^{k}\Lambda(m_i) \Bigr)^2}{{m}^2}-\sum_{\substack{m\geq 1\\p|m}} \frac{\Bigl(\sum\limits_{m=m_1\cdots m_k} \prod\limits_{i=1}^{k}\Lambda(m_i) \Bigr)^2}{{m}^2}\nonumber
\\&=&
\sum_{m\geq 1} \frac{\Bigl(\sum\limits_{m=m_1\cdots m_k} \prod\limits_{i=1}^{k}\Lambda(m_i) \Bigr)^2}{{m}^2}+\Ocal\left(\frac{(\log p)^{2k}}{p^2}\right),
\end{eqnarray}
where we used Lemma~\ref{lem4}. 
\end{proof}
\section{Proof of Theorems~\ref{Thm1} and \ref{cor1}}
\label{section5}

Let $q\geq 2$.
We consider the function 
\begin{equation*}
G_q(s)=\sum_{\substack{\chi \mkern3mu \mathrel{\textsl{mod}} \mkern3mu q}}\left(\frac{L^{\prime}(s+it_0, \chi)}{L(s+it_0, \chi)}\right)^k\left(\frac{L^\prime(s-it_0, \overline{\chi})}{L(s-it_0, \overline{\chi})}\right)^k
\end{equation*}
where $\chi$ runs over all Dirichlet characters modulo $q$. When $\sigma>1$, using the fact that 
\begin{equation*}
\frac{L^\prime(s, \chi)}{L(s, \chi)}=\sum_{n\geq 1}\frac{\chi(n)\Lambda(n)}{n^s},
\end{equation*}
one can write the function $G_q(s)$ as 
\begin{equation*}
G_q(s)=
\sum_{\substack{\chi \mkern3mu \mathrel{\textsl{mod}} \mkern3mu q}}\sum_{\substack{m_1\cdots m_k \geq 1 \\n_1\cdots n_k \geq 1
 }}\frac{\prod\limits_{i=1}^{k} \Lambda(m_i)\chi(m_i) \prod\limits_{i=1}^{k} \Lambda(n_i)\bar{\chi}(n_i)}{\left(\prod\limits_{i=1}^{k} m_in_i\right)^s \left(\prod\limits_{i=1}^{k} m_i\right)^{it_0}\left(\prod\limits_{i=1}^{k}n_i\right)^{-it_0}}.
\end{equation*}
The proof of our theorems relies on two distinct evaluations of the quantity:
\begin{equation}
\label{eq018}
S_q(X)=\frac{1}{2\pi i}\int\limits_{3-i\infty}^{3+i\infty}G_q(s)X^{s-1}\Gamma(s-1)\, ds.
\end{equation}
We write the integrand of the right-hand side of the above as $f(s)$.
\subsection{The first evaluation of $S_q(X)$} It relies on the formula $e^{-y}=\frac{1}{2i\pi} \int_{2-i\infty}^{2+i\infty}y^{-s}\Gamma(s)\, ds$ (valid for positive $y$) and on the use of Proposition~\ref{Thm1}. We readily find that
\begin{eqnarray*}
S_q(X)&=& \frac{1}{2\pi i}\sum_{\chi \mkern3mu \mathrel{\textsl{mod}} \mkern3mu q}\sum_{\substack{m_1\cdots m_k \geq 1 \\n_1\cdots n_k \geq 1
 }}\frac{\prod\limits_{i=1}^{k} \Lambda(m_i)\chi(m_i) \prod\limits_{i=1}^{k} \Lambda(n_i)\bar{\chi}(n_i)}{\left(\prod\limits_{i=1}^{k} m_i\right)^{1+it_0}\left(\prod\limits_{i=1}^{k}n_i\right)^{1-it_0}}
 \int_{2-i\infty}^{2+i\infty}\Bigl(\frac{X}{\prod\limits_{i=1}^{k} m_in_i}\Bigr)^s\Gamma(s)\, ds
 \\&=&
 \varphi(q)\sum_{\substack{m, n \geq 1 \\ m\equiv n \mkern3mu \mathrel{\textsl{mod}} \mkern3mu q \\ (q, mn)=1} }\frac{\left(\sum\limits_{m=m_1\cdots m_k} \Lambda(m_1)\cdots \Lambda(m_k) \sum\limits_{n=n_1\cdots n_k} \Lambda(n_1)\cdots \Lambda(n_k)\right)}{m^{1+it_0}n^{1-it_0}}\, e^{-mn/X}.
\end{eqnarray*}
Thanks to Proposition~\ref{mainpro}, we get 
\begin{equation}
\label{eq019}
S_q(X)
=
\varphi(q)\sum_{ \substack{m\geq 1 \\(m,q)=1}} \frac{\left(\sum\limits_{{m}=m_1\cdots m_k}\prod_{i=1}^{k}\Lambda(m_i)\right)^{2}}{{m}^2}+Y
\end{equation}
with
\begin{equation}\label{est_Y}
Y=\Ocal\left(\frac{\varphi(q)}{q}(\log X)^{2k+2}+
\frac{\varphi(q)(\log X)^{2k}}{\sqrt{X}}\right).
\end{equation}

\subsection{The second evaluation of $S_q(X)$} 
From Proposition~\ref{thm3}, we note that the following regions
$$ \mathcal{D}_1= \left\{\sigma \geq 1-\frac{c}{\log (q(|t+t_0|+2))}\right\}$$
and 
$$\mathcal{D}_2= \left\{\sigma \geq 1-\frac{c}{\log (q(|t-t_0|+2))}\right\},$$
are zero-free regions of the functions $L(s+it_0, \chi)$ and $L(s-it_0, \overline{\chi})$ respectively, except for the possible Siegel zero $\beta_1$. 
Then, for any Dirichlet character $\chi \pmod q$ and $T\geq 2$, we see that the region  
\begin{equation*}
\mathcal{D}_3=\left\{ \sigma \geq 1-\frac{c}{\log (q(T+|t_0|+2))}, \quad |t|\leq T\right\}
\end{equation*}
is a zero-free region of the both functions $L(s+it_0, \chi)$ and $L(s-it_0, \overline{\chi})$, except for the possible zeros $\beta_1\pm it_0$ (see Figure~\ref{Fig2}).
\begin{figure}[h]
\centering
\includegraphics[width=9cm]{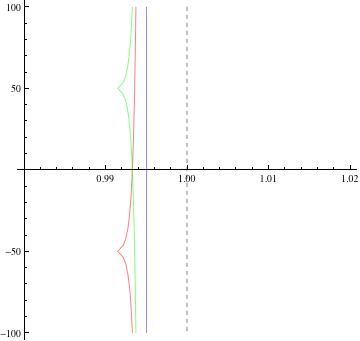}
\caption{The regions $ \mathcal{D}_1$ (red),  $\mathcal{D}_2$ (green) and $ \mathcal{D}_3$ (purple).}
\label{Fig2}
\end{figure}

Now, let $A(c_1)=1-c_1/\log (q(T+|t_0|+2))$ with $0<c_1<c\; (<c_0/2)$, and
shift the part $|t|\le T$ of the path of integration in Eq.~\eqref{eq018} to the line 
segment $\sigma+it$ defined with 
$\sigma=A(c_1)$ and $|t|\le T$. 
We choose $c_1$ so that $\beta_1$ (if exists in the region \eqref{region}) satisfies the inequality
\begin{align}\label{c_1_condition}
|\beta_1-A(c_1)|\geq \frac{c_1}{10\log (q(T+|t_0|+2))}.
\end{align}
Put
$$f_{\chi,t_0}(s)=\left(\frac{L'(s+it_0,\chi)}{L(s+it_0,\chi)}\right)^k\left(\frac{L'(s-it_0, \overline{\chi})}{L(s-it_0, \overline{\chi})}\right)^k\Gamma(s-1) X^{s-1},$$
then $f(s)=\sum_{\chi}f_{\chi,t_0}(s)$. Let $\mathcal{C}_T$ denote the closed contour that consists of line segments joining the points $3-iT$, $3+iT$, $A(c_1)+iT$ and $A(c_1)-iT$ shown Figure~\ref{Fig3}, that is  $\mathcal{C}_T=I_1\cup I_2 \cup I_3 \cup I_4$ with
\begin{itemize}
\item $I_1$: The line segment from $3-iT$ to $3+iT$, 
\item $I_2$: The line segment from $3+iT$ to $A(c_1)+iT$,
\item $I_3$: The line segment from $A(c_1)+iT$ to $A(c_1)-iT$, 
\item $I_4$: The line segment from $A(c_1)-iT$ to $3-iT$.
\end{itemize} 
By Eq.~\eqref{eq018}, we note that all the possibilities of the poles of the function $ f_{\chi, t_0}(s)$ occurring inside $\mathcal{C}_T$ are as follows: 
\begin{itemize}
\item $s_1$: a pole at $1$, for any $t_0$ and for any $ \chi$, 
\item $s_2, s_3$: two poles at $ 1+ it_0$ and $1-it_0$ respectively, of order $k$, when $ \chi=\chi_0$ and $t_0\neq 0$,  
\item $s_4, s_5$: two possible poles at $\beta_1+ it_0$ and $\beta_1-it_0$ respectively, of order $k$, when $\chi =\chi_1$ and $t_0\neq 0$,
\item $s_6$: a possible pole of order $2k$ at $s=\beta_1$ when $\chi = \chi_1$ and $t_0=0$.
\end{itemize}
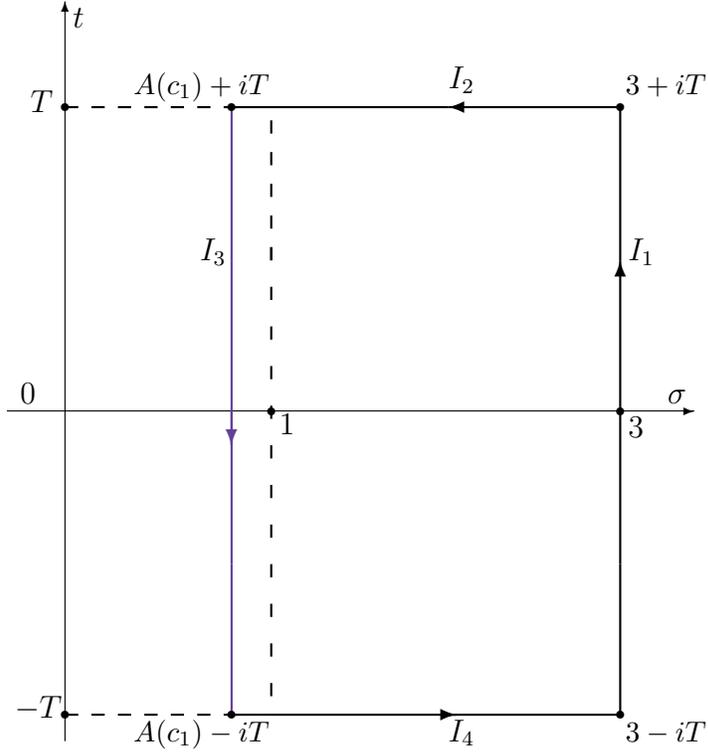
\begin{figure}
\begin{center}
\begin{picture}(270,280)
\put(-5,125){\vector(1,0){260}}
\put(245,128){$\sigma$}
	
\put(17,0){\vector(0,1){280}}
\put(20,270){$t$}
		
\thicklines
\put(227,10){\line(0,1){57}}
\put(227,67){\line(0,1){58}}			
\put(227,125){\vector(0,1){57}}
\put(227,182){\line(0,1){58}}
\put(230,182){$I_1$}
		
\put(227,240){\vector(-1,0){65}}
\put(162,240){\line(-1,0){83}}
\put(162,247){$I_2$}

{\color{RoyalPurple}{\put(80,182){\line(0,1){58}}
\put(80,182){\vector(0,-1){70}}
\put(80,67){\line(0,1){58}}
\put(80,67){\line(0,-1){57}}}}
\put(68,182){$I_3$}
	
\put(80,10){\vector(1,0){85}}
\put(162,10){\line(1,0){65}}
\put(162,0){$I_4$}
		
\put(0,128){$0$}
\put(227,125){\circle*{3}}
\put(230,115){$3$}
\put(95,125){\circle*{3}}
\put(98,116){$1$}

\put(17,240){\circle*{3}}
\put(4,238){$T$}

\put(17,10){\circle*{3}}
\put(-2,9){$-T$}

\put(227,240){\circle*{3}}
\put(229,245){\small{$3+iT$}}

\put(227,10){\circle*{3}}
\put(229,0){\small{$3-iT$}}

\put(80,10){\circle*{3}}
\put(43,0){\small{$A(c_1)-iT$}}

\put(80,240){\circle*{3}}
\put(43,245){\small{$A(c_1)+iT$}}

\multiput(95,182)(0,12){5}{\line(0,1){5}}
\multiput(95,182)(0,-15){12}{\line(0,1){5}}

\multiput(17,240)(10,0){7}{\line(1,0){5}}
\multiput(17,10)(10,0){7}{\line(1,0){5}}
			
		
			
\end{picture}
\end{center}
\caption{The contour $C_T$ in the complex plane.}
\label{Fig3}
\end{figure}
\subsubsection{The calculus of residues.}
\begin{itemize}
\item[\underline{Pole $s_1$}:] We distinguish two cases depending on $t_0$. The first case is  when $t_0\neq 0$. We observe that the function $f_{\chi,t_0}(s)$ has a pole at $s=1$ of order $1$. Then, one finds that
\begin{equation}
\label{eq020}
\res( f_{\chi,t_0}(s); 1 )=
\left(\frac{L^{\prime}(1+it_0,\chi)}{L(1+it_0,\chi)}\right)^k\left(\overline{\frac{L^\prime(1+it_0, \chi)}{L(1+it_0, \chi)}}\right)^k.
\end{equation}
The second case is when $t_0=0$. For $\chi\neq \chi_0$, the function $f_{\chi,t_0}(s)$ has again a pole at $s=1$ of order $1$. Then
\begin{equation*}
\res( f_{\chi,0}(s); 1 )=
\left(\frac{L^{\prime}(1, \chi)}{L(1, \chi)}\right)^k\left(\frac{L^\prime(1, \overline{\chi})}{L(1, \overline{\chi})}\right)^k.
\end{equation*}
As for $\chi=\chi_0$, the function $f_{\chi_0,0}(s)$ has a pole at $s=1$ of order $2k+1$ and the residue of our function at this point is calculated as follows: Taking $s'=s-1$, we find that  
\begin{equation}
\label{eq022}
X^{s'}=\sum_{n=0}^{\infty}M_{n,0}(X)s'^n
\end{equation}
and that 
\begin{equation}
\label{eq023}
s'\Gamma(s')=\Gamma(s'+1)=\sum_{n=0}^{\infty}N_{n,0} s'^n, 
\end{equation}
where 
\begin{equation}
\label{eq024}
M_{n,0}(X)=(\log X)^n/n!, \qquad N_{n,0}=\Gamma^{(n)}(1)/n!.
\end{equation} 
Using the fact that $ L(s,\chi_0)=\zeta(s)\prod_{p|q}\left(1-\frac{1}{p^s}\right)$, we write
\begin{equation*}
s'\frac{L'(s'+1, \chi_0)}{L(s'+1, \chi_0)}=
s'\frac{\zeta'(s'+1)}{\zeta (s'+1)} +s'\sum_{p|q}\frac{\log p}{p^{s'+1}-1}.
\end{equation*}
Thanks to Lemma~\ref{lem7} and Lemma~\ref{lem5} with $a=1$, we get 
\begin{eqnarray}
\label{eq00}
s'\frac{L'(s'+1, \chi_0)}{L(s'+1, \chi_0)}&=&
\sum_{n=0}^{\infty} E_{n} s'^{n}+\sum_{n=0}^{\infty} F_{n, 1}s'^{n+1}\nonumber
\\&=&
\sum_{n=0}^{\infty} E_{n} s'^{n}+\sum_{n=1}^{\infty} F_{n-1, 1}s'^{n}\nonumber
\\&=&\sum_{n=0}^{\infty} H_{n} s'^{n},
\end{eqnarray}
where $H_0=E_{0}$ and $H_{n}=E_{n}+ F_{n-1, 1}$ for $n\geq 1$. Here the coefficients $  E_n, F_{n, 1}$ are defined by Eqs.~\eqref{eq09} and \eqref{eq02} respectively. Using the properties of power series, one finds that 
\begin{equation}
\label{eq025}
\left(s'\frac{L'(s'+1, \chi_0)}{L(s'+1, \chi_0)}\right)^{2k}=\Bigl(\sum_{n=0}^{\infty} H_{n} s'^{n}\Bigr)^{2k}=\sum_{n=0}^{\infty} \tilde{H}_{n} s'^{n}, 
\end{equation}
where 
\begin{equation}
\label{eq026}
\tilde{H}_{n}=\sum_{n=n_1+ n_2+\cdots+ n_{2k}}H_{n_1} H_{n_2}\cdots H_{n_{2k}}=\Ocal_{k}\left( Q^{2k}\right).
\end{equation}
By Eqs.~\eqref{eq022},~\eqref{eq023} and~\eqref{eq025}, we infer
\begin{eqnarray*}
\res( f_{\chi_0, 0}(s); 1)&=&\frac{1}{(2k)!}\lim_{s\rightarrow 1}\frac{d^{2k}}{ds^{2k}}\left[(s-1)^{2k+1} f_{\chi_0, 0}(s)\right]
\\&=&
\frac{1}{(2k)!}\lim_{s'\rightarrow 0}\frac{d^{2k}}{(ds')^{2k}}\left[s'^{2k+1} f_{\chi_0, 0}(s'+1)\right]
\\&=&
\frac{1}{(2k)!}\lim_{s'\rightarrow 0}\frac{d^{2k}}{(ds')^{2k}}\left[ \sum_{n=0}^{\infty}J_n(X) s'^n\right],
\end{eqnarray*}
where the coefficients $J_n(X) $ are determined by multiplying the above three series together and via the properties of power series, namely 
\begin{equation}
\label{eq028}
J_n(X)=\sum_{n=n_1+ n_2+ n_3} M_{n_1,0}(X)N_{n_2,0}\tilde{H}_{n_3}, 
\end{equation}
where $ M_{n_1, 0}(X), N_{n_2, 0}$ and $\tilde{H}_{n_3}$ are defined by Eqs.~\eqref{eq024} and \eqref{eq026} respectively. 
Therefore, we get
\begin{equation}
\label{eq027}
\res( f_{\chi_0,0}(s); 1)=J_{2k}(X)=\Ocal\left(Q^{2k}(\log X)^{2k}\right). 
\end{equation}
From Eqs.~\eqref{eq020} and \eqref{eq027}, we write 
\begin{equation}
\label{eq0028}
\res( f(s); 1)=
\begin{cases}
\sum\limits_{\chi \mkern3mu \mathrel{\textsl{mod}} \mkern3mu q }\left|\frac{L^{\prime}(1+it_0, \chi)}{L(1+it_0, \chi)}\right|^{2k}, &  t_0\neq 0;\cr
\sum\limits_{\substack{\chi \mkern3mu \mathrel{\textsl{mod}} \mkern3mu q\\ \chi\neq \chi_0} }\left|\frac{L^{\prime}(1,\chi)}{L(1,\chi)}\right|^{2k}+J_{2k}(X), &  t_0=0.
\end{cases}
\end{equation}
\item[\underline{Pole $s_2$}:] For $\chi=\chi_0$ and $t_0\neq 0$, the function $f_{\chi_0, t_0}(s)$ has a pole at $s=1+it_0$ of order $k$. Taking $s'=s-1-it_0$, we write each term of $f_{\chi_0, t_0}(s)$ as follows 
\begin{equation}
\label{eq029}
X^{s-1}=X^{it_0}e^{s'\log X}=\sum_{n=0}^{\infty}M_{n, t_0}(X) s'^n,
\end{equation}
\begin{equation}
\label{eq030}
\Gamma(s-1)=\Gamma(s'+it_0)=\sum_{n=0}^{\infty}N_{n, t_0} s'^n, 
\end{equation}
where 
\begin{equation}
\label{eq031}
M_{n, t_0}(X) = X^{it_0}\frac{(\log X)^n}{n!},
\qquad\qquad  N_{n, t_0} =\frac{\Gamma^{(n)}(it_0)}{n!}.
\end{equation}
Again using the fact that $ L(s,\chi_0)=\zeta(s)\prod_{p|q}\left(1-\frac{1}{p^s}\right)$, we find that
\begin{eqnarray*}
\frac{L'(s+it_0, \chi_0)}{L(s+it_0, \chi_0)}&=&\frac{L'(s'+1+2it_0, \chi_0) }{L(s'+1+2it_0, \chi_0) }
\\&=&
\frac{\zeta'(s'+1+2it_0)}{\zeta (s'+1+2it_0)} +\sum_{p|q}\frac{\log p}{p^{s'+1+2it_0}-1}.
\end{eqnarray*}
Using Lemma~\ref{lem6} with $s_0=1+2it_0$ and Lemma~\ref{lem5} with $a=1+2it_0$, the above function is written in the form 
\begin{equation*}
\frac{L'(s+it_0, \chi_0)}{L(s+it_0, \chi_0)}=
\sum_{n=0}^{\infty}K_{n,t_0} s'^n,
\end{equation*}
where 
\begin{equation}
\label{eq033}
K_{n,t_0} = C_{n,1+2it_0} + F_{n,1+2it_0}.
\end{equation}
Here $C_{n,1+2it_0}$ and $ F_{n,1+2it_0}$  are defined in Eqs.~\eqref{eq04} and \eqref{eq02} respectively. Thus, we get 
\begin{equation}
\label{eq034}
\left(\frac{L'(s+it_0, \chi_0)}{L(s+it_0, \chi_0)}\right)^k=
\Bigl(\sum_{n=0}^{\infty}K_{n,t_0} s'^n\Bigr)^k=\sum_{n=0}^{\infty}\tilde{K}_{n,t_0} s'^n,
\end{equation}
where 
\begin{eqnarray}
\label{eq035}
 \tilde{K}_{n,t_0}&=&\sum_{n=n_1+n_2+\cdots+n_{k}}K_{n_1,t_0}K_{n_2,t_0}\cdots K_{n_k,t_0}.
\end{eqnarray}
From Eqs.~\eqref{eq05} and \eqref{eq02} we have
\begin{align*}
K_{n,t_0}=\left\{
\begin{array}{lll}
O(|t_0|^{-n-1}+Q) & \;{\rm if} \;& 0<|t_0|\leq 1,\\
O((\log(|t_0|+2))^{n+1}+Q) &\; {\rm if}\; & |t_0|>1.
\end{array}\right.
\end{align*}
Therefore if $0<|t_0|\leq 1$, 
\begin{align*}
\tilde{K}_{n,t_0}&\ll \sum_{n=n_1+n_2+\cdots+n_{k}}\left(\frac{1}{|t_0|^{n_1+1}}+Q\right)
\cdots \left(\frac{1}{|t_0|^{n_k+1}}+Q\right)\\
&\ll Q^k+\sum_{l=1}^k \sum_{0\leq n_1+n_2+\cdots+n_{l}\leq n}
\frac{Q^{k-l}}{|t_0|^{n_1+\cdots+n_l+l}}.
\end{align*}
Each term in the sum is
$$
\ll \frac{Q^{k-l}}{|t_0|^{n+l}}
\leq\max\left\{\frac{1}{|t_0|^{n+k}},\frac{Q^{k-1}}{|t_0|^{n+1}}\right\},
$$
and hence
\begin{align}\label{estim1-tilde-K}
\tilde{K}_{n,t_0}\ll_{n,k} Q^k+\frac{Q^{k-1}}{|t_0|^{n+1}}+\frac{1}{|t_0|^{n+k}}
\qquad (|t_0|\leq 1).
\end{align}
Similarly,
\begin{align}\label{estim2-tilde-K}
\tilde{K}_{n,t_0}\ll_{n,k} Q^k+Q^{k-1}(\log(|t_0|+2))^{n+1}+(\log(|t_0|+2))^{n+k}
\qquad (|t_0|> 1).
\end{align}
Next, using Eq.~\eqref{eq00}, we have
\begin{eqnarray}
\label{eq036}
\left((s-1-it_0)\frac{L'(s-it_0, \chi_0)}{L(s-it_0, \chi_0)}\right)^k&=&
\left(s'\frac{L'(s'+1, \chi_0)}{L(s'+1, \chi_0)}\right)^k \nonumber
\\&=&\Bigl( \sum_{n=0}^{\infty}H_ns'^{n}\Bigr)^k=\sum_{n=0}^{\infty}\tilde{\tilde{H}}_ns'^n, 
\end{eqnarray}
where $\tilde{\tilde{H}}_n$ is defined by Eq.~\eqref{eq026} with  $2k$ replaced by $k$ and hence $\tilde{\tilde{H}}_n=\Ocal\left( Q^k\right)$.
From Eqs.~\eqref{eq029}, \eqref{eq030}, \eqref{eq034} and \eqref{eq036}, we therefore get 
\begin{eqnarray*}
\res( f_{\chi_0, t_0}(s); 1+it_0)&=&\frac{1}{(k-1)!}\lim_{s\rightarrow 1+it_0}\frac{d^{k-1}}{ds^{k-1}}\left[(s-1-it_0)^k G_q(s)\Gamma(s-1) X^{s-1}\right]
\\&=&
\frac{1}{(k-1)!}\lim_{s'\rightarrow 0}\frac{d^{k-1}}{(ds')^{k-1}}\left[s'^k G_q(s'+1+it_0)\Gamma(s'+it_0) X^{s'+it_0}\right]
\\&=&
\frac{1}{(k-1)!}\lim_{s'\rightarrow 0}\frac{d^{k-1}}{(ds')^{k-1}}\left[ \sum_{n=0}^{\infty}L_{n, t_0}(X) s'^n\right]
\\&=& L_{k-1, t_0}( X),
\end{eqnarray*}
where
\begin{equation}
\label{eq039}
L_{n, t_0}(X)=\sum_{n=n_1+n_2+n_3+n_4} M_{n_1, t_0}(X)N_{n_2, t_0}\tilde{K}_{n_3,t_0}\tilde{\tilde{H}}_{n_4}, 
\end{equation}
where $M_{n_1, t_0}(X)$ and $N_{n_2, t_0}, \tilde{K}_{n_3,t_0}$ and $\tilde{\tilde{H}}_{n_4}$ are given by Eqs.~\eqref{eq031}, \eqref{eq035} and \eqref{eq026} respectively.
Recall the Stirling formula  
\begin{equation}
\label{eq00000}
\Gamma(\sigma+it )=\sqrt{2\pi} \,(1+|t|)^{\sigma-1/2} e^{-\pi|t|/2 }\left(1+\Ocal\left(1/|t|\right)\right).
\end{equation}
Then we see that $\Gamma^{(n)}(it_0)=O_n(\exp(-C_1|t_0|))$ (with a certain absolute $C_1>0$) 
for $|t_0|>1$, while it is
$=O_n(|t_0|^{-n-1})$ for $0<|t_0|\leq 1$.
Therefore we find the following
evaluation of $L_{k-1, t_0}(X)$.
First, if $|t_0|>1$, from \eqref{eq031} and \eqref{estim2-tilde-K} we have
$$
L_{k-1, t_0}(X)\ll_k (\log X)^{k-1}e^{-C_2|t_0|} Q^{2k},
$$
where $0<C_2<C_1$.
Secondly, if $0<|t_0|\leq 1$, then
$$
L_{k-1, t_0}(X)\ll \sum_{k-1=n_1+n_2+n_3+n_4}(\log X)^{n_1}\frac{1}{|t_0|^{n_2+1}}
\left(Q^k+\frac{Q^{k-1}}{|t_0|^{n_3+1}}+\frac{1}{|t_0|^{n_3+k}}\right)Q^k,
$$
but the factors $(\log X)^{n_1}|t_0|^{-n_2}$, $(\log X)^{n_1}|t_0|^{-n_2-n_3}$ are
estimated by $\leq (\log X)^{k-1}+|t_0|^{-k+1}$, hence
\begin{align*}
L_{k-1, t_0}(X)&\ll_k  \left((\log X)^{k-1}+\frac{1}{|t_0|^{k-1}}\right)
\frac{Q^k}{|t_0|}\left(Q^k+\frac{Q^{k-1}}{|t_0|}+\frac{1}{|t_0|^k}\right)\\
&\ll \left((\log X)^{k-1}+\frac{1}{|t_0|^{k-1}}\right)
\frac{Q^k}{|t_0|}\left(Q^k+\frac{1}{|t_0|^k}\right).
\end{align*}
Therefore, we now conclude that 
\begin{align}
\label{eq038}
&\res( f_{\chi_0, t_0}(s); 1+it_0)=L_{k-1,t_0}(X)\\&\quad=\left\{
\begin{array}{lll}
\Ocal\left( (\log X)^{k-1}e^{-C_2|t_0|} Q^{2k}\right) &\; {\rm if}\; & |t_0|>1,\\
\displaystyle{\Ocal\left(\left((\log X)^{k-1}+\frac{1}{|t_0|^{k-1}}\right)
\frac{Q^k}{|t_0|}\left(Q^k+\frac{1}{|t_0|^k}\right) 
\right)} & \;{\rm if}\; & 0<|t_0|\leq 1.
\end{array}\right.\notag
\end{align}
\item[\underline{Pole $s_3$}:] For $\chi=\chi_0$ and $t_0\neq 0$, the function $f_{\chi_0, t_0}(s)$ has a pole at $s=1-it_0$ of order $k$. We calculate the residue of $f(s)$ at the point $1-it_0$ similar to that in the previous case. We get
\begin{equation}
\label{eq040}
\res( f_{\chi_0, t_0}(s); 1-it_0)=L_{k-1, -t_0}(X),
\end{equation}
where $L_{n, -t_0}(X)$ is defined by Eq.~\eqref{eq039} and satisfies the same estimate 
as \eqref{eq038}.
\item[\underline{Pole $s_4$}:] For $\chi= \chi_1$ and $t_0\neq 0$, the function $f_{\chi_1, t_0}(s)$ has a (possible) pole at $s=\beta_1+it_0$ of order $k$. Putting $s'=s-\beta_1-it_0$,  we write each term of $f_{\chi_1, t_0}(s)$ as follows 
\begin{equation}
\label{eq041}
X^{s-1}=X^{\beta_1-1+it_0}e^{s'\log X}=\sum_{n=0}^{\infty}\tilde{M}_{n,t_0}(X)s'^n, 
\end{equation}
\begin{equation}
\label{eq042}
\Gamma(s-1)=\Gamma(s_3+\beta_1-1+it_0)=\sum_{n=0}^{\infty}\tilde{N}_{n,t_0}s'^n, 
\end{equation}
where 
\begin{equation}
\label{eq043}
\tilde{M}_{n, t_0}(X)=X^{\beta_1 -1+it_0}\frac{(\log X)^n}{n!}, \quad 
\tilde{N}_{n,t_0}=\frac{\Gamma^{(n)}(\beta_1-1+it_0)}{n!}. 
\end{equation}
Using Lemma \ref{lem8}, we find that 
\begin{eqnarray}
\label{eq0000}
s'\frac{L'(s'+\beta_1, \chi_1)}{L(s'+\beta_1, \chi_1)} 
=
s'\left( \frac{1}{s'}+\sum_{n\geq 0} P_n s'^n\right)
=
\sum_{n\geq 0} P_{n-1} s'^{n}, 
\end{eqnarray}
where $P_{-1}=1$ and $P_n$ is defined in Lemma~\ref{lem8}. Hence, we get
\begin{equation}
\label{eq047}
\left(s'\frac{L'(s'+\beta_1, \chi_1)}{L(s'+\beta_1, \chi_1)}\right)^k=
\sum_{n=0}^{\infty}\tilde{P}_n s'^n,
\end{equation}
where 
\begin{equation}
\label{eq048}
\tilde{P}_n=\sum_{n=n_1+\cdots+n_{k}}P_{n_1-1}\cdots P_{n_k-1}=\Ocal\left((\log q)^n \right).
\end{equation}
On the other hand, we use Lemma~\ref{lem9} to write
\begin{equation*}
\frac{L'(s+it_0, \chi_1) }{L(s+it_0, \chi_1) }=\frac{L'(s'+\beta_1+2it_0, \chi_1) }{L(s'+\beta_1+2it_0, \chi_1) }=\sum_{n=0}^{\infty}Q_{n}{s'}^n.
\end{equation*}
This leads at once to 
\begin{equation}
\label{eq053}
\left(\frac{L'(s+it_0, \chi_1)}{L(s+it_0, \chi_1)}\right)^k=
\sum_{n=0}^{\infty}\tilde{Q}_{n} s'^n,
\end{equation}
where 
\begin{equation}
\label{eq054}
 \tilde{Q}_{n}= \sum_{n=n_1+\cdots+n_k}Q_{n_1}\cdots Q_{n_k}=\Ocal\left( \left(\log\left( q(|t_0|+2)\right)\right)^{n+k}+\frac{1}{|t_0|^{n+k}}\right).
\end{equation}
From Eqs.~\eqref{eq041}, \eqref{eq042}, \eqref{eq047} and \eqref{eq053}, we therefore get 
\begin{eqnarray*}
\res( f_{\chi_1, t_0}(s); s_4)&=&\frac{1}{(k-1)!}\lim_{s\rightarrow \beta+it_0}\frac{d^{k-1}}{ds^{k-1}}\left[(s-\beta_1-it_0)^{k} f_{\chi_1, t_0}(s)\right]
\\&=&
\frac{1}{(k-1)!}\lim_{s'\rightarrow 0}\frac{d^{k-1}}{(ds')^{k-1}}\left[s'^{k} f_{\chi_1, t_0}(s'+\beta_1+it_0)\right]
\\&=&
\frac{1}{(k-1)!}\lim_{s'\rightarrow 0}\frac{d^{k-1}}{(ds')^{k-1}} \sum_{n=0}^{\infty}R_{n, t_0}(q, X) s'^n
\\&=&
R_{k-1,t_0}(q,X),
\end{eqnarray*}
where 
\begin{equation}
\label{eq056}
R_{n,t_0}(q,X)=\sum_{n=n_1+n_2+n_3+n_4} \tilde{M}_{n_1,t_0}(X)\tilde{N}_{n_2,t_0}\tilde{P}_{n_3}\tilde{Q}_{n_4}, 
\end{equation}
with $ \tilde{M}_{n_1,t_0}(X)$ and $\tilde{N}_{n_2,t_0}, \tilde{P}_{n_3}$ and $\tilde{Q}_{n_4}$ defined by Eqs.~\eqref{eq043},~\eqref{eq048} and~\eqref{eq054} respectively. 
If $|t_0|>1$, then $\Gamma^{(n)}(\beta_1-1+it_0)=O_n(\exp(-C_3|t_0|))$ (with a certain
absolute $C_3>0$), and hence
\begin{align*}
R_{k-1,t_0}(q,X)&\ll X^{\beta_1-1}e^{-C_4|t_0|}\sum_{k-1=n_1+n_2+n_3+n_4} 
(\log X)^{n_1}(\log q)^{n_3+n_4+k}\\
&\ll X^{\beta_1-1}e^{-C_4|t_0|}\left((\log X)^{2k-1}+(\log q)^{2k-1}\right)
\end{align*}
(where $0<C_4<C_3$).
If $0<|t_0|\leq 1$, then 
$\Gamma^{(n)}(\beta_1-1+it_0)\ll_n |\beta_1-1+it_0|^{-n-1}\leq |t_0|^{-n-1}$.
Therefore
\begin{align*}
R_{k-1,t_0}(q,X)&\ll X^{\beta_1-1}\sum_{k-1=n_1+n_2+n_3+n_4} 
(\log X)^{n_1}\frac{1}{|t_0|^{n_2+1}}(\log q)^{n_3}\frac{1}{|t_0|^{n_4+k}}\\
&\ll \frac{X^{\beta_1-1}}{|t_0|^{k+1}}\left((\log X)^{k-1}+
\frac{1}{|t_0|^{k-1}}+(\log q)^{k-1}\right).
\end{align*}
Therefore we now obtain
\begin{align}
\label{eq055}
&\res( f_{\chi_1, t_0}(s); \beta_1+it_0)=R_{k-1,t_0}(q,X)\\
&\quad =\left\{
\begin{array}{lll}
O\left(X^{\beta_1-1}e^{-C_4|t_0|}\left((\log X)^{2k-1}+(\log q)^{2k-1}\right)\right)
& \;{\rm if}\; & |t_0|>1,\\
\displaystyle{O\left(\frac{X^{\beta_1-1}}{|t_0|^{k+1}}\left((\log X)^{k-1}+
\frac{1}{|t_0|^{k-1}}+(\log q)^{k-1}\right)\right)} & \;{\rm if}\; & 0<|t_0|\leq 1.
\end{array}\right.\notag
\end{align}
\item[\underline{Pole $s_5$}:] Similarly, we get 
\begin{equation}
\label{eq057}
\res(f_{\chi_1, t_0}(s); \beta_1-it_0)=R_{k-1,-t_0}(q, X),
\end{equation}
where $R_{k-1,-t_0}(q, x)$ is defined by Eq.~\eqref{eq056} and satisfies the same estimate as
\eqref{eq055}.
\item[\underline{Pole $s_6$}:] For $ \chi=\chi_1$ and $t_0=0$, the function $f_{\chi_1,t_0}(s)$ has a (possible) pole of order $2k$ at $s=\beta_1$. Putting $s'=s-\beta_1$, we find that  
\begin{equation*}
(s-\beta_1)\frac{L'(s, \chi_1)}{L(s, \chi_1)}
=s'\frac{L'(s'+\beta_1, \chi_1)}{L(s'+\beta_1, \chi_1)},
\end{equation*}
where the right-hand side is equal to $\sum_{n\geq 0}P_{n-1}s'^n$ by
Eq.~\eqref{eq0000}.
 Hence, we get
\begin{equation}
\label{eq058}
\left((s-\beta_1)\frac{L'(s, \chi_1)}{L(s,\chi_1)}\right)^{2k}=
\sum_{n=0}^{\infty}\tilde{\tilde{P}}_n s'^n,
\end{equation}
where $\tilde{\tilde{P}}_n$ is given by Eq.~\eqref{eq048} with $k$ replaced by $2k$.
From Eqs.~\eqref{eq041}, \eqref{eq042} and \eqref{eq058}, we therefore get 
\begin{eqnarray*}
\res( f_{\chi_1, t_0}(s); \beta_1)&=&\frac{1}{(2k-1)!}\lim_{s\rightarrow \beta_1}\frac{d^{2k-1}}{ds^{2k-1}}\left[(s-\beta_1)^{2k} G_q(s)\Gamma(s-1) X^{s-1}\right]
\\&=&
\frac{1}{(2k-1)!}\lim_{s'\rightarrow 0}\frac{d^{2k-1}}{(ds')^{2k-1}}\left[s'^{2k} G_q(s'+\beta_1)\Gamma(s'+\beta_1-1) X^{s'+\beta_1-1}\right]
\\&=&
\frac{1}{(2k-1)!}\lim_{s'\rightarrow 0}\frac{d^{2k-1}}{(ds')^{2k-1}}\left[ \sum_{n=0}^{\infty}Y_{n}(q, X) s'^n\right],
\end{eqnarray*}
where 
\begin{equation}
\label{eq061}
Y_{n}(q, X)=\sum_{n=n_1+n_2+n_3} \tilde{M}_{n_1,0}(X)\tilde{N}_{n_2,0}\tilde{\tilde{P}}_{n_3}, 
\end{equation}
with $ \tilde{M}_{n_1,0}(X)$ and $\tilde{N}_{n_2,0}$ and $\tilde{\tilde{P}}_{n_3}$ being defined by Eqs.~\eqref{eq043} and~\eqref{eq048} respectively. 
Since $\Gamma^{(n)}(\beta_1-1)=O((1-\beta_1)^{-n-1})$, 
we have 
\begin{align}
\label{eq060}
&\res( f_{\chi_1, t_0}(s); \beta_1)=Y_{2k-1}(q, X)\\
&\qquad=\Ocal\left(X^{\beta_1-1}((\log X)^{2k-1}+(1-\beta_1)^{-2k}+
(\log q)^{2k-1})\right).\notag
\end{align}
\end{itemize}
Consequently, we find from  Eqs.~\eqref{eq027}, ~\eqref{eq0028}, \eqref{eq038}, \eqref{eq040}, \eqref{eq055}, \eqref{eq057}  and \eqref{eq060} that 
\begin{equation}
\label{eq062}
\sum_{i=1}^{6}\res(f(s); s_i)=
\begin{cases}
\sum\limits_{\chi \mkern3mu \mathrel{\textsl{mod}} \mkern3mu q }\left|\frac{L^{\prime}(1+it_0, \chi)}{L(1+it_0, \chi)}\right|^{2k}+Z_{k, t_0}(q, X), &  t_0\neq 0;\cr
\sum\limits_{\substack{\chi \mkern3mu \mathrel{\textsl{mod}} \mkern3mu q\\ \chi\neq \chi_0} }\left|\frac{L^{\prime}(1, \chi)}{L(1,\chi)}\right|^{2k}+Z_{k,0}(q,X),&  t_0=0,
\end{cases}
\end{equation}
where 
\begin{align}
\label{est_Z}
&Z_{k,t_0}(q, X)=L_{k-1, t_0}(q,X)+L_{k-1, -t_0}(q,X)+\delta_1 R_{k-1,t_0}(q,X)+
\delta_1 R_{k-1,-t_0}(q,X)
\\&=\left\{
\begin{array}{ll}
\Ocal\left((\log X)^{k-1}e^{-C_2|t_0|}Q^{2k}+
\delta_1 X^{\beta_1-1}e^{-C_4|t_0|}((\log X)^{2k-1}+(\log q)^{2k-1})\right) & (|t_0|>1),
\\
\displaystyle{\Ocal\left(\left((\log X)^{k-1}+\frac{1}{|t_0|^{k-1}}\right)
\frac{Q^k}{|t_0|}\left(Q^k+\frac{1}{|t_0|^k}\right) 
\right)} & (0<|t_0|\leq 1)
\end{array}\right.\notag
\end{align}
(note that when $0<|t_0|\leq 1$ the right-hand side of \eqref{eq055} is absorbed into
the right-hand side of \eqref{eq038})
and 
\begin{equation}
\label{est_Z0}
Z_{k,0}(q, X)=J_{2k}(X)+\delta_1 Y_{2k-1}(q,X)=\Ocal\left((\log X)^{2k}Q^{2k}+
\delta_1 X^{\beta_1-1}(1-\beta_1)^{-2k}\right).
\end{equation}
\subsubsection{The evaluation of the integration on $I_i$  }
Now, we are going to estimate the integration on $I_i$ where $i\in \{2, 3, 4\}$. Denote 
$$ J_i=\frac{1}{2\pi i}\int_{I_i}G_q(s)X^{s-1}\Gamma(s-1) \, ds.$$
On these paths, in view of Eqs.~\eqref{pro31}--\eqref{pro33}, we have
\begin{equation*}
\frac{L'(s\pm it_0,\chi)}{L(s\pm it_0,\chi)}\ll \log (q(T+|t_0|+2))
\end{equation*}
on $I_i$,
for any $\chi$ modulo $q$ (in the case $\chi=\chi_1$, we use \eqref{c_1_condition}). 
First consider the integral on $I_3$. Then
$\left|X^{s-1}\right| \leq X^{A(c_1)-1}$,
and hence
\begin{eqnarray*}
J_3
&\ll& 
\varphi(q)\left(\log \left(q(T+|t_0|+2)\right)\right)^{2k}X^{A(c_1)-1}\int\limits_{A(c_1)-iT}^{A(c_1)+iT}\left|\Gamma(s-1)\right|\, |ds|
\\&\ll &
\varphi(q)\left(\log \left(q(T+|t_0|+2)\right)\right)^{2k}X^{A(c_1)-1}\int\limits_{-T}^{T}\left|\Gamma\left(A(c_1)-1+it\right)\right|\, dt. 
\end{eqnarray*}
From Eq.~\eqref{eq00000} we obtain 
\begin{equation*}
\left|\Gamma\left(A(c_1)-1+it\right)\right| \ll  (1+|t|)^{A(c_1)-\frac{3}{2}}\, e^{-\pi |t|/2},
\end{equation*}
and so
\begin{equation}
\label{J3}
J_3 
\ll 
\varphi(q)\left(\log \left(q(T+|t_0|+2)\right)\right)^{2k}X^{A(c_1)-1}.
\end{equation}
Now we calculate the integrals along the
horizontal segments. Since the integrand has the same
absolute value at conjugate points, it suffices to consider only the upper segment $t = T$. On this segment we have the estimate
\begin{equation*}
J_2
\ll
\varphi(q)\left(\log \left(q(T+|t_0|+2)\right)\right)^{2k}\int\limits_{A(c_1)}^{3}\left|\Gamma(\sigma-1+iT)\right|X^{\sigma-1}\, d\sigma.
\end{equation*}
Again, using Eq.~\eqref{eq00000}, we get 
\begin{eqnarray}
\label{J2}
J_2
&\ll&
    \varphi(q)\left(\log\left(q(T+|t_0|+2)\right)\right)^{2k}X^{-1}(1+T)^{-\frac{3}{2}}\ e^{-\frac{\pi T}{2}}\int\limits_{A(c_1)}^{3}\left((1+T)X\right)^{\sigma}\, d\sigma \nonumber
\\&\ll&
\varphi(q)\left( \log \left(q(T+|t_0|+2)\right)\right)^{2k}\frac{X^{2}(1+T)^{\frac{3}{2}}\ e^{-\frac{\pi T}{2}}}{\log \left( (1+T)X\right)},
\end{eqnarray}
and $J_4$ can be estimated similarly.
\subsection{The conclusion}\label{subsec5-3}
On the half-lines $\sigma=3$ and $|t| \geq T$, we have 
\begin{equation*}
\int\limits_{\substack{\sigma=3\\  |t| \geq T}}G_q(s)X^{s-1}\Gamma(s-1)\, ds 
\ll 
 \varphi(q)X^2 \int\limits_{t\geq T}\left|\Gamma(2+it)\right|\, dt.
\end{equation*}
Again applying \eqref{eq00000}, we get 
\begin{equation}
\label{eq066}
\int\limits_{\substack{\sigma=3\\ |t| \geq T}}G_q(s)X^{s-1}\Gamma(s-1)\, ds 
\ll 
\varphi(q) X^2 (1+T)^{3/2}\, e^{-\pi T/2}.
\end{equation}
Therefore, by combining Eqs.~\eqref{eq062}, \eqref{J3}, \eqref{J2} and \eqref{eq066}, we obtain
\begin{equation*}
S_q(X)=
\begin{cases}
\sum\limits_{\chi \mkern3mu \mathrel{\textsl{mod}} \mkern3mu q }\left|\frac{L^{\prime}(1+it_0, \chi)}{L(1+it_0, \chi)}\right|^{2k}+Z_{k, t_0}(q, X)+W, &  t_0\neq 0;\cr
\sum\limits_{\substack{\chi \mkern3mu \mathrel{\textsl{mod}} \mkern3mu q\\ \chi\neq \chi_0} }\left|\frac{L^{\prime}(1, \chi)}{L(1, \chi)}\right|^{2k}+Z_{k,0}(q,X)+W,&  t_0=0,
\end{cases}
\end{equation*}
where $W$ is estimated by 
\begin{multline}
\label{est_W}
  \Ocal\left(\varphi(q)\left(\log \left(q(T+|t_0|+2)\right)\right)^{2k}\left\{X^{A(c_1)-1}+\frac{X^2(1+T)^{\frac{3}{2}}e^{-\frac{\pi T}{2}}}{\log\left((1+T)X\right)} \right\}+\varphi(q)X^2(1+T)^{\frac{3}{2}}e^{-\frac{\pi T}{2}}\right).    
\end{multline}
Now we combine Eq.~\eqref{eq019} and the above formula.
The remaining task is to evaluate 
$Z_{k,t_0}(q,X)+W+Y$, under some suitable choices of 
parameters $T$ and $X$.
Our choices are $T=q$ and $X=\exp \left( \lambda(\log q)^{2}\right)$ (where $\lambda$
is a large positive number).

First consider $W$.    Under the above choices, we have
\begin{align}\label{est000}
X^{A(c_1)-1}=\exp\left(-\frac{c_1\lambda(\log q)^2}{\log(q(q+|t_0|+2))}\right)
\leq \exp\left(-\frac{c_1\lambda(\log q)^2}{2\log(q+|t_0|+2)}\right),
\end{align}
which is, when $q\geq |t_0|+2$, 
\begin{align*}
\leq \exp\left(-\frac{c_1\lambda(\log q)^2}{2\log(2q)}\right)
\leq \exp\left(-\frac{c_1\lambda(\log q)^2}{4\log q}\right)
= \exp\left(-\frac{c_1\lambda}{4}\log q\right).
\end{align*}
We choose $\lambda$ sufficiently large: $\lambda\geq\max\{ 4/c_1, 2/\log 2\}$.
Then from the above we see that $X^{A(c_1)-1}\leq \exp(-\log q)=q^{-1}$.
Since the factor $e^{-\pi T/2}=e^{-\pi q/2}$ is very small with respect to $q$, 
from \eqref{est_W} and \eqref{est000} we obtain
\begin{align}\label{estfinal_W}
W=O\left(\varphi(q)\left(\log \left(q(q+|t_0|+2)\right)\right)^{2k}
\exp\left(-\frac{B_1(\log q)^2}{\log(q+|t_0|+2)}\right)\right).
\end{align}
with $B_1=c_1\lambda/2\geq 2$.
In particular, when $t_0=0$, we have
\begin{align}\label{estfinal_W_0}
W=O\left(\frac{\varphi(q)}{q}(\log q)^{2k}\right).
\end{align}

Next, we have
\begin{align*}
Y\ll \frac{\varphi(q)}{q}(\log q)^{4k+4}+\frac{\varphi(q)(\log q)^{4k}}
{\exp((\lambda/2)(\log q)^2)}.
\end{align*}
By the assumption $\lambda\geq 2/\log 2$ we have
$$\exp((\lambda/2)(\log q)^2)\geq \exp((\lambda/2)\log 2\log q)\geq \exp(\log q)=q,$$ 
so
\begin{align}\label{estfinal_Y}
Y=O\left(\frac{\varphi(q)}{q}(\log q)^{4k+4}\right).
\end{align}
Lastly, we find
\begin{align}
\label{estfinal_Z}
&Z_{k,t_0}(q):=Z_{k,t_0}(q,\exp(\lambda(\log q)^2))
\\&=\left\{
\begin{array}{ll}
\Ocal\left((\log q)^{2k-2}e^{-B_2|t_0|}Q^{2k}\right) & (|t_0|>1),
\\
\displaystyle{\Ocal\left(\left((\log q)^{2k-2}+\frac{1}{|t_0|^{k-1}}\right)
\frac{Q^k}{|t_0|}\left(Q^k+\frac{1}{|t_0|^k}\right) 
\right)} & (0<|t_0|\leq 1)
\end{array}\right.\notag
\end{align}
where $B_2=\min\{C_2,C_4\}$, and 
\begin{align}
\label{estfinal_Z0}
&Z_{k,0}(q):=Z_{k,0}(q,\exp(\lambda(\log q)^2))
\\&\qquad=
\Ocal\left((\log q)^{4k}Q^{2k}+
\delta_1\exp\left(-(1-\beta_1)\lambda(\log q)^2\right)(1-\beta_1)^{-2k}\right).\notag
\end{align}
Collecting all of the above estimates, 
we arrive at the assertions of Theorems \ref{Thm1} and \ref{cor1}.
\begin{Remark}
\label{rem1}
\emph{Using Proposition~\ref{mainpro2} instead of Proposition~\ref{mainpro}, the same proof works for $q=p$ a prime number and then one can show that the condition $(m,q)=1$ in the main term in Theorems \ref{Thm1} and \ref{cor1} is omitted. }
\end{Remark}
\section{Proof of Theorem \ref{Thm2}}

Now we proceed to the proof of Theorem~\ref{Thm2}. We deduce the existence of $\mu$ by the general solution to the Stieltjes moment problem and the unicity by the criterion of Carleman. First, we define the ``problem of moments'' which was showed up in the work of Stieltjes (1894-1895). 

\subsection{Problem of moments}
The problem of moments is to find a bounded non-decreasing function $\psi(x)$ in the interval $[0, \infty)$ such that its ''moments" $\int_{0}^{\infty} x^k \, d\psi(x)$, $k=0, 1, 2, \cdots$, have a prescribed set of values 
\begin{equation}
\label{measure}
\int_{0}^{\infty}x^k\, d\psi(x)=\mu_k, \qquad k=1, 2, \cdots.
\end{equation}
This problem  was first raised and solved by Stieltjes (1895-1895) for non-negative measures. He proved in~\cite{Stieltjes} that Eq.~\eqref{measure} has a solution if and only if the following determinants are non-negative:
\begin{equation*}
\Delta_k=
  \left| {\begin{array}{ccccc}
   \mu_0 & \mu_1 & \mu_2 & \cdots& \mu_k\\
   \mu_1 & \mu_2& \mu_3& \cdots& \mu_{k+1}\\
   \mu_2& \mu_3& \mu_4& \cdots& \mu_{k+2}\\
  \vdots& \vdots& \vdots& \ddots& \vdots\\
   \mu_k& \mu_{k+1}& \mu_{k+2}& \cdots& \mu_{2k}
  \end{array} } \right|= \left| \mu_{i+j}\right|_{i, j=0}^{k},\qquad \qquad k=0, 1, 2, \cdots,
\end{equation*}
\begin{equation*}
  \Delta^{*}_k=
    \left| {\begin{array}{ccccc}
     \mu_1 & \mu_2 & \mu_3 & \cdots& \mu_{k+1}\\
     \mu_2 & \mu_3& \mu_4& \cdots& \mu_{k+2}\\
     \mu_3& \mu_4& \mu_5& \cdots& \mu_{k+3}\\
     \vdots& \vdots& \vdots& \ddots& \vdots\\
     \mu_{k+1}& \mu_{k+2}& \mu_{k+3}& \cdots& \mu_{2k+1}
    \end{array} } \right| =\left| \mu_{i+j+1}\right|_{i, j=0}^{k},\qquad k=0, 1, 2, \cdots,
\end{equation*}
The following proposition provides the necessary and sufficient condition for the existence of a solution of the Stieltjes moment problem.
\begin{prop}
\label{Stieltjes4}
A necessary and sufficient condition that the Stieltjes moment problem defined by the sequence of moments $\{\mu_k\}_{k=0}^{\infty}$ shall have a solution is that the functional $\mu(P)$ is non-negative, that is $$\mu(P)=\sum_{j=0}^{k}\mu_jx_j\geq 0,$$
for any polynomial 
 $$P(u)=x_0+x_1u+\cdots +x_ku^k, \qquad (x_0, x_1, \cdots, x_k\in, \mathbb{R})$$ which is non-negative for all $u\geq 0$.
\end{prop}
\begin{proof}
A proof of this result can be found in \cite[Theorem 1.1]{Sh}. 
\end{proof}
Now, consider the following two polynomials  
\begin{equation*}
 Q_k(u)=\left( x_0+x_1u+\cdots+x_ku^k\right)^2,
\end{equation*}
\begin{equation*}
 R_k(u)=u\left( x_0+x_1u+\cdots+x_ku^k\right)^2.
\end{equation*}
We note that $Q_k(u)\geq 0$ and $R_k(u)\geq 0$ for $u\in [0, \infty)$ and $k=0, 1, 2, \cdots.$ Using the fact that any polynomial $P(u)\geq 0$ for $u\geq 0$ can be written in the form $p_1(u)^2+up_2(u)^2$ with certain polynomials $p_1(u)$ and $p_2(u)$ (see the footnote in \cite[page 6]{Sh}), we translate the condition in Proposition~\ref{Stieltjes4} into the following condition 
\begin{equation}
\label{condition1}
 \mu(P)\geq 0  \qquad  \text{if and only if} \qquad  \mu(Q_k)\geq 0 \quad  \text{and} \quad \mu(R_k)\geq 0,  
\end{equation} 
for all $k=0, 1, 2, \cdots, .$
On the other hand, $Q_k(u)$ and $R_k(u)$ are of the form 
\begin{eqnarray*}
    &&Q_k(u)=\sum_{i,j=0}^{k}x_ix_j u^{i+j},
    \\&& R_k(u)=\sum_{i,j=0}^{k}x_ix_j u^{i+j+1},
\end{eqnarray*}
so, it follows that 
\begin{eqnarray*}
    &&\mu(Q_k)=\sum_{i,j=0}^{k}x_ix_j \mu_{i+j},
    \\&& \mu(R_k)=\sum_{i,j=0}^{k}x_ix_j \mu_{i+j+1}.
\end{eqnarray*}
From the theory of quadratic forms it is well known that
\begin{equation*}
    \label{condition2}
\mu(Q_k)\geq 0 \ \text{and}\ \mu(R_k)\geq 0 \quad \text{if and only if} \quad \Delta_k=\left| \mu_{i+j}\right|_{i, j=0}^{k}\geq 0\ \text{and}\ \Delta^{*}_k=\left| \mu_{i+j+1}\right|_{i, j=0}^{k}\geq 0.
\end{equation*}
From the above, we deduce the following result:
\begin{cor}
\label{stieltjes10}
A necessary and sufficient condition that the Stieltjes moment problem defined by the sequence of moments $\{\mu_k\}_{k=0}^{\infty}$ shall have a solution is that 
$$
\Delta_k=\left| \mu_{i+j}\right|_{i, j=0}^{k}\geq 0\quad \text{and}\quad \Delta^{*}_k=\left| \mu_{i+j+1}\right|_{i, j=0}^{k}\geq 0,$$
for all $k=0, 1, 2, \cdots.$
\end{cor}

\subsection{Proof of Theorem \ref{Thm2}}
\subsubsection*{Existence of $\mu$}

We define the measure $\mu_q$, depending on $t_0$, by  $\mu_q\left( [0,v]\right):=D_q(v, t_0)$ where $D_q(v, t_0)$ is given by Eq.~\eqref{D}. Then,
we have $\mu_q$ is non-negative and $\mu_q\left([0, \infty)\right)=1$. Setting 
\begin{eqnarray*}
m_k(q,t_0)&:=&\int_{0}^{\infty}v^{k}\, d\mu_q(v)\\
&=&\frac{1}{\varphi(q)}\sum_{\chi\hspace{-0.2cm}\mod q}^{\quad \prime}\left|\frac{L'(1+it_0)}{L(1+it_0)} \right|^{2k},
\end{eqnarray*}
where $\sum^{\prime}$ runs over all Dirichlet characters $\chi$ modulo $q$ except the principal character in the case $t_0=0$. By Corollary~\ref{stieltjes10}, we get 
$$
\Delta_k(q,t_0)=\left| m_{i+j}\right|_{i, j=0}^{k}\geq 0\quad \text{and}\quad \Delta^{*}_k(q,t_0)=\left| m_{i+j+1}\right|_{i, j=0}^{k}\geq 0.$$
On the other hand, from Theorems~\ref{Thm1} and \ref{cor1}, $m_k(q, t_0)$ can be written as follows
\begin{equation*}
m_k(q, t_0)=M_k(q, t_0)+N_k(q, t_0), 
\end{equation*}
where 
\begin{equation*}
M_k(q, t_0)=\sum_{\substack{m\geq 1\\ (m,q)=1}}\frac{\left(\sum\limits_{m=m_1 m_2\cdots m_k}\Lambda(m_1)\cdots\Lambda(m_k)\right)^2}{m^2}
\end{equation*}
and $N_k(q, t_0)$ is the error term which tends to 0 as $q\to\infty$.
Therefore, we get 
\begin{equation*}
    \Delta_k(q, t_0)=\left| M_{i+j}(q, t_0)\right|_{i, j=0}^{k}+E_k(q, t_0)\geq 0
\end{equation*}
and 
\begin{equation*}
    \Delta^{*}_k(q, t_0)=\left| M_{i+j+1}(q, t_0)\right|_{i, j=0}^{k}+E^*_k(q, t_0)\geq 0,
\end{equation*}
where $E_k(q, t_0)$ and $E^*_k(q, t_0)$ are error terms which tend to $0$ as $q\longrightarrow \infty.$
Now, we assume that $q=p$ is a prime number. By Remark~\ref{rem1}, $m_k(p, t_0)$ is rewritten as 
\begin{equation*}
m_k(p, t_0)=M_k(t_0)+N_k(p, t_0),  
\end{equation*}
where 
\begin{equation*}
M_k(t_0)=\sum_{m\geq 1}\frac{\left(\sum\limits_{m=m_1 m_2\cdots m_k}\Lambda(m_1)\cdots\Lambda(m_k)\right)^2}{m^2},
\end{equation*}
which is independent of $p$. By letting $p$ tend to infinity  follows that
\begin{equation}
\label{condition}
    \left| M_{i+j}(t_0)\right|_{i, j=0}^{k}\geq 0 \quad \text{and} \quad \left| M_{i+j+1}(t_0)\right|_{i, j=0}^{k}\geq 0.
\end{equation}
We again apply Corollary~\ref{stieltjes10} to find a measure $\mu=\mu(t_0)$ such that 
\begin{equation*}
\lim_{p\longrightarrow \infty}\frac{1}{p-1}\sum_{\chi\hspace{-0.2cm}\mod p}^{\quad \prime}\left|\frac{L'(1+it_0)}{L(1+it_0)} \right|^{2k}
=
\int_{0}^{\infty}v^{k}\, d\mu(v),
\end{equation*}
because the left-hand side is equal to $M_k(t_0).$ 
\subsubsection*{Uniqueness of $\mu$}
In order to complete our proof, it remains to show that $\mu$ is unique. 
There are several sufficient conditions for uniqueness. In our proof we shall use Carleman's condition~\cite{Pro}, which states that the solution is unique if
\begin{equation*}
 \sum_{k\geq 1}\frac{1}{M_k^{\frac{1}{2k}}}
= \infty.
\end{equation*}
We use Lemma~\ref{lem4} to get 
\begin{equation*}
M_k \leq \sum_{m\geq 1} \frac{(\log m)^{2k}}{m^2}.
\end{equation*}
Now, we notice that 
\begin{eqnarray*}
\sum_{m\geq 1} \frac{(\log m)^{2k}}{m^2} 
&\ll  & 
\int_{1}^{\infty} \frac{(\log t)^{2k}}{t^2}\, dt 
\\&= & \int_{0}^{\infty} u^{2k}\, e^{-u}\, du
\ll  \Gamma(2k+1)
= (2k)!.
\end{eqnarray*}
Then, we have 
\begin{equation}
\label{eq21}
M_k \ll (2k)!.
\end{equation}
Therefore, we get
\begin{equation*}
\sum_{k\geq 1}\frac{1}{M_k^{\frac{1}{2k}}} \gg \sum_{k\geq 1}\left(\frac{1}{(2k)!}\right)^{\frac{1}{2k}}
= \infty
\end{equation*}
It follows that the condition of Carleman is checked and thus the function $\mu$ is unique. This completes the proof. 

\section{Scripts}
We present here an easier GP script for computing the values $\left|(L^\prime /L)(1, \chi)\right|$. In this loop, we use the Pari package " ComputeL" written by Tim Dokchitser to compute values of $L$-functions and its derivative. This package is available on-line at 
 \begin{center}
 {\tt www.maths.bris.ac.uk/\~{}matyd/}
 \end{center}
 On this base we write the next script. the authors would like to thank Professor Olivier Ramar\'e for helping us in writing it. We simply plot Figure~\ref{Fig1} via \\
 
  \begin{verbatim}
   read("computeL"); /* by Tim Dokchitser */
  default(realprecision,28); 
  {run(p=37)=
     local(results, prim, avec);
     prim = znprimroot(p);
     results = vector(p-2, i, 0);
     for(b = 1, p-2,
        avec = vector(p,k,0);
        for (k = 0, p-1, avec[lift(prim^k)+1]=exp(2*b*Pi*I*k/(p-1)));
        conductor = p; 
        gammaV    = [1];
        weight    = b%2; 
        sgn       = X;
        initLdata("avec[k%p+1]",,"conj(avec[k%p+1])"); 
        sgneq = Vec(checkfeq());
        sgn   = -sgneq[2]/sgneq[1]; 
        results[b] = abs(L(1,,1)/L(1));
           \\print(results[b]);
        );
     return(results);
  }

  {goodrun(borneinf, bornesup)=
     forprime(p = borneinf, bornesup,
              print("------------------------");
              print("p = ",p);
              print(vecsort(run(p))));}
  \end{verbatim}

\subsection*{Acknowledgement} 
The first author is supported by ``JSPS KAKENHI Grant Number: JP25287002''. The second author is supported by the Austrian Science
Fund (FWF): Projects F5507-N26, and F5505-N26 which are parts
of the special Research Program  ``Quasi Monte
Carlo Methods : Theory and Application''. Part of this work was also done while she was supported by the Japan
Society for the Promotion of Science (JSPS) ``Overseas researcher under Postdoctoral Fellowship of JSPS''. The authors would like to thank Professor J{\"o}rn Steuding for helpful feedback, and acknowledges fruitful discussions with Dr. Ade Irma Suriajaya. The authors also express their gratitude to the anonymous referee for a lot of useful comments, especially for pointing out inaccuracies included in the original version of the manuscript.

\medskip\noindent {\footnotesize Kohji Matsumoto: Graduate School of Mathematics, Nagoya University, Furo-cho, Chikusa-ku, Nagoya, Aichi 464-8602, Japan. \\
e-mail: {\tt kohjimat@math.nagoya-u.ac.jp}}

\medskip\noindent {\footnotesize Sumaia Saad Eddin: Institute of Financial Mathematics and Applied Number Theory, Johannes Kepler Universit\"at Linz, Altenbergerstrasse 69, 4040 Linz, Austria\\
e-mail: {\tt sumaia.saad\_eddin@jku.at}}
\end{document}